\documentclass[aap]{imsart}
\usepackage{graphicx} 
\usepackage{mathtools, algorithm, algpseudocode}
\usepackage{amsthm}
\usepackage{enumitem}
\usepackage{bbold}
\usepackage{amsmath, amssymb}
\usepackage{amsfonts}
\usepackage{mathrsfs}
\usepackage{chngcntr}
\usepackage{bigints}
\usepackage{apptools}
\usepackage{subcaption}
\usepackage{caption}
\AtAppendix{\counterwithin{theorem}{section}}
\usepackage[english]{babel}
\usepackage[autostyle, english = american]{csquotes}
\MakeOuterQuote{"}
\theoremstyle{definition}
\newtheorem{definition}{Definition}[section]
\newtheorem{theorem}{Theorem}

\newtheorem{proposition}[theorem]{Proposition}

\newtheorem*{remark}{Remark}
\newtheorem{example}{Example}[section]
\newtheorem{lemma}[theorem]{Lemma}
\usepackage{enumitem}
\usepackage{comment}
\usepackage{color}
\usepackage{float}

\newcommand{\indep}{\perp \!\!\! \perp}

\RequirePackage{amsthm,amsmath,amsfonts,amssymb}
\RequirePackage[numbers,sort&compress]{natbib}
\RequirePackage[colorlinks,citecolor=blue,urlcolor=blue]{hyperref}
\RequirePackage{graphicx}

\startlocaldefs
\theoremstyle{plain}

\theoremstyle{remark}

\endlocaldefs

\begin{document}

\begin{frontmatter}
\title{Nonparametric Regression in Dirichlet Spaces: A Random Obstacle Approach}
\runtitle{Dirichlet Space Regression}

\begin{aug}
\author[A]{\fnms{Prem}~\snm{Talwai} \thanks{[\textbf{Corresponding author indication should be put in the Acknowledgment section if necessary.}]}}
\and
\author[B]{\fnms{David}~\snm{Simchi-Levi}\orcid{0000-0000-0000-0000}}
\address[A]{Operations Research Center, MIT}

\address[B]{Institute for Data, Systems, and Society, MIT}
\end{aug}

\begin{abstract}
In this paper, we consider nonparametric estimation over general Dirichlet metric measure spaces. Unlike the more commonly studied reproducing kernel Hilbert space, whose elements may be defined pointwise, a Dirichlet space typically only contain equivalence classes, i.e. its elements are only unique almost everywhere. This lack of pointwise definition presents significant challenges in the context of nonparametric estimation, for example the classical ridge regression problem is ill-posed. In this paper, we develop a new technique for renormalizing the ridge loss by replacing pointwise evaluations with certain \textit{local means} around the boundaries of obstacles centered at each data point. The resulting renormalized empirical risk functional is well-posed and even admits a representer theorem in terms of certain equilibrium potentials, which are truncated versions of the associated Green function, cut-off at a data-driven threshold. We demonstrate that the renormalized ridge estimator is rate-optimal, and derive an adaptive upper bound on its convergence rate that highlights the interplay between the analytic, geometric, and probabilistic properties of the Dirichlet form. Our framework notably does not require the smoothness of the underlying space, and is applicable to both manifold and fractal settings. To the best of our knowledge, this is the first paper to obtain optimal, out-of-sample convergence guarantees in the framework of general metric measure Dirichlet spaces.  
\end{abstract}



\end{frontmatter}

\section{Introduction}

We consider the classical problem of nonparametric estimation where we observe i.i.d samples $\mathcal{D} \equiv \{(X_i, Y_i)\}_{i = 1}^n \subset \mathcal{X} \times \mathbb{R}$ from the model:
\begin{equation*}
    Y = f^{*}(X) + \epsilon
\end{equation*}
where $\epsilon$ is independent of $X$, with mean-zero and finite variance. Our goal is to estimate $f^{*}$ using only the sample dataset $\mathcal{D}$. \\

This estimation problem has a storied history in mathematical statistics \cite{wainwright2019high}, typically with the primary focus on relating the consistency of a sample estimator to the smoothness properties of $f^{*}$. Here, the regularity of $f^{*}$ is typically characterized by its membership in a function class $\mathcal{H}$, and one aims to design an estimator that optimizes an expected minimax risk over the entire function class. \\

While nonparametric estimation over \textit{supercritical} function classes, i.e. those that contain continuous functions, has been well-studied, \textit{subcritical} classes, where elements are only defined almost everywhere, have been rarely examined. Indeed, in arguably the most commonly studied supercritical setting, $\mathcal{H}$ is a reproducing kernel Hilbert space (RKHS), where pointwise evaluations of its elements may be represented by certain feature vectors in $\mathcal{H}$. This characterization, known as the reproducing property, serves as the cornerstone of modern kernel methods, and underscores their powerful analytic and computational tractability. \\

In this paper, we study nonparametric estimation over a family of subcritical function classes, known as Dirichlet spaces (defined formally in section \ref{Preliminaries}). Dirichlet spaces are characterized by their association with an intrinsic bilinear form known as a \textit{Dirichlet form}, which enjoys a Markovian property and serves as the critical bridge between the probabilistic and analytic characterizations of the associated Markovian semigroup \cite{fukushima2010dirichlet}. Over the past three decades, the theory of Dirichlet forms has emerged as a powerful, cohesive machinery for studying various stochastic models, particularly with non-smooth data, on fractals, and in infinite-dimensional settings \cite{albeverio2015dirichlet}. In recent years, Dirichlet forms have also experienced a renaissance in the machine learning community due to their tractable characterization of the intrinsic geometry of both point clouds and their continuum limits \cite{hoffmann2022spectral, calder2022improved, trillos2021geometric, cheng2022convergence}. Our general setting of Dirichlet metric measure spaces enables the unified analysis of a rich family of examples, including fractal spaces, domains of uniformly elliptic divergence forms on $\mathbb{R}^d$, first-order Sobolev spaces on Riemannian manifolds with Ricci curvature bounded from below, polynomial growth Lie groups and their homogeneous spaces, $\text{RCD}^{*}(K, N)$ spaces in the theory of optimal transport, Carnot-Caratheodory spaces and many more \cite{coulhon2012heat, coulhon2020gradient}. \\

Unlike an RKHS, whose elements may be defined pointwise, a Dirichlet space typically only contains \textit{equivalence classes}, i.e. its elements are only unique almost everywhere. Indeed, consider the canonical example of a Dirichlet space, the first-order Euclidean Sobolev space $\mathcal{H} = \mathbb{H}_{0}^1(\mathbb{R}^d)$. For $d \geq 3$, $\mathbb{H}_0^1(\mathbb{R}^d)$ does not enjoy a Sobolev embedding into the continuous functions $\mathcal{C}_0(\mathbb{R}^d)$. This presents a significant challenge in the design of sample estimators, as many classical approaches, such as ridge regression are ill-posed. Indeed, consider the thin-plate estimator, i.e. the solution to:
\begin{equation}
\label{eq: Thin-Plate}
    \text{arg}\min_{g \in \mathbb{H}^1_{0}} \frac{1}{n}\sum_{i = 1}^n (Y_i - g(X_i))^2 + \lambda\int \|\nabla g(x)\|^2 dx 
\end{equation}
When $d \geq 3$, for any $\delta > 0$, we can build a test function $g_{\delta} \in  \mathbb{H}^1_{0}$ using bump functions localized around $\{X_i\}_{i = 1}^n$, such that $g_{\delta}(X_i) = Y_i$ and:
\begin{equation*}
    \int \|\nabla g_{\delta}(x)\|^2 dx < \delta
\end{equation*}
In other words, $\{g_{\frac{1}{n}}\}_{n = 1}^{\infty}$ forms a minimizing sequence that achieves zero loss in \eqref{eq: Thin-Plate}. However, it may be readily seen that $g_{\frac{1}{n}} \to 0$ in $L^2$, which clearly is an inconsistent estimator of any nonzero mean function. \\

Due to the ill-posedness of \eqref{eq: Thin-Plate}, nonparametric techniques in subcritical Euclidean Sobolev spaces typically wholly discard a continuous analysis, and instead favor a discrete approach \cite{green2021minimax, green2023minimax, garcia2020maximum} that is grounded in the spectral analysis of a graph Laplacian constructed on the data points. While such approaches, including Laplacian smoothing \cite{green2021minimax} and PCR \cite{green2023minimax}, have gained popularity due to their tractable implementation, the resulting graph estimators \textit{are only defined on the data sample} and critically do not provide out-of-sample generalization. Moreover, the convergence analysis of these estimators relies strongly on the smooth/Euclidean structure of the underlying domain, and does not extend to the setting of general metric spaces, considered here. \\

In this paper, we develop a new approach that aims to directly renormalize the continuous-space ridge loss function:

\begin{equation}
\label{eq: Ridge Regression}
    \min_{g \in \mathcal{H}} \frac{1}{n}\sum_{i = 1}^n (Y_i - g(X_i))^2 + \lambda\mathcal{E}(g, g)
\end{equation}

in order to produce a well-posed objective. Here, $\mathcal{H}$ is an (extended) transient Dirichlet space and $\mathcal{E}(\cdot, \cdot)$ is the associated Dirichlet form. Our approach leads to a \textit{globally defined} estimator on $\mathcal{X}$, that enjoys optimal, adaptive convergence guarantees.  In the following section, we provide a brief overview of our approach which will later be elaborated in section \ref{Random Obstacle Regression}. 
\subsection{Random Obstacle Renormalization}
\label{Random Obstacle Preview}
To tackle the ill-posedness of \eqref{eq: Ridge Regression}, we consider a slightly relaxed problem. Namely, we replace the pointwise evaluations in \eqref{eq: Ridge Regression} by their ``local means'' on the boundaries of \textit{obstacles} surrounding the data points. We will see that these local means \textit{are} continuous functionals on $\mathcal{H}$, hence enabling a pointwise representation of the elements of $\mathcal{H}$. More specifically, we consider the obstacles:

\begin{equation}
\label{eq: Obstacle}
    \mathcal{O}_{i, n} = \{y: G(X_i, y) \geq \gamma_n\}
\end{equation}
where $G$ is the \textit{Green's function} for $\mathcal{H}$ on $L^2$ (see \eqref{eq: Green Definition} in section \ref{Preliminaries}) and $\gamma_n$ is a \textit{cutoff} constant depending on $n = |\mathcal{D}|$ (the number of data points). \\

In other words, we pick our obstacles as \textit{level sets} of the Green function $G(X_i, \cdot)$ with singularity at the data point $X_i$. On a high level, $G(\cdot, \cdot)$ plays a similar role to the reproducing kernel of an RKHS, however crucially $\mathbf{G(x, \cdot) \not \in \mathcal{H}}$ (actually $G(x, \cdot) \not \in L^2$; it is not even defined at $x$). \\

Instead for each obstacle, we consider its $\textit{equilibrium potential}$, i.e. the unique solution to the following variational problem:

\begin{equation}
\label{eq: Potential Definition}
e_{i, n} = \text{arg} \min_{g \in \mathcal{L}(\mathcal{O}_{i, n})} \mathcal{E}(g, g)
\end{equation}
where:
\begin{equation*}
    \mathcal{L}(\mathcal{O}_{i, n}) \equiv \{g \in \mathcal{H}: g \geq \mathbb{1}_{\mathcal{O}_{i, n}}\}
\end{equation*}
For each obstacle $\mathcal{O}_{i, n}$ there is a unique \textit{equilibrium measure} $\nu_{i, n}$ concentrated on $\partial \mathcal{O}_{i, n}$ such that for all $h \in \mathcal{H}$:

\begin{equation*}
    \mathcal{E}(h, e_{i, n}) = \int h(x)d\nu_{i, n}(x)
\end{equation*}

and $\nu_{i, n}(\partial \mathcal{O}_{i, n}) = \gamma_n^{-1} = \text{cap}(\mathcal{O}_{i, n})$ (the $\mathcal{E}$-capacity of $\mathcal{O}_{i, n}$, see section \ref{Capacity Preliminaries}). Hence, after normalization, we see:

\begin{equation}
\label{eq: Capacitary Mean}
    \mathcal{E}(h, \gamma_n e_{i, n}) = \frac{1}{\nu_{i, n}(\partial \mathcal{O}_{i, n})}\int h(x)d\nu_{i, n}(x)
\end{equation}
is our desired \textit{capacitary} mean. Let us denote the mean operator:
\begin{equation*}
    P_{i, n} h = \frac{1}{\nu_{i, n}(\partial \mathcal{O}_{i, n})}\int h(x)d\nu_{i, n}(x)
\end{equation*}
Then, we perturb \eqref{eq: Ridge Regression} to the following:

\begin{equation}
\label{eq: Capacitary Ridge Regression}
    \hat{f}_{D, \lambda} = \text{arg}\min_{g \in \mathcal{H}} \frac{1}{n}\sum_{i = 1}^n (Y_i - P_{i, n}g)^2 + \lambda\mathcal{E}(g, g)
\end{equation}

whose optimal solution can be expressed as:

\begin{equation*}
    \hat{f}_{D, \lambda} = \Big(\frac{\gamma^2_n}{n}\sum_{i = 1}^n e_{i, n} \otimes e_{i, n} + \lambda \Big)^{-1} \frac{\gamma_n}{n}\sum_{i = 1}^n Y_i e_{i, n}
\end{equation*}

or equivalently:
\begin{equation}
\label{eq: Representer Theorem}
    \hat{f}_{D, \lambda} = \sum_{i = 1}^n c_i \gamma_n e_{i, n}
\end{equation}
where $\mathbf{c} = [c_1, \ldots, c_n]^T \in \mathbb{R}^n$ is given by:
\begin{align}
    \mathbf{c} & = (\mathbf{G}_n + n\lambda \mathbf{I})^{-1}\mathbf{y} \nonumber \\
    \mathbf{y} & = [Y_1, \ldots, Y_n]^T \nonumber \\
    (\mathbf{G}_n)_{i, j} & = \gamma^2_n\mathcal{E}(e_{i, n}, e_{j, n}) \label{eq: Gram Matrix}
\end{align}
In other words, our perturbed problem \eqref{eq: Capacitary Ridge Regression} \textbf{enjoys a representer theorem \eqref{eq: Representer Theorem} in terms of the equilibrium potentials $e_{i, n}$}. In section \ref{Random Obstacle Regression}, we will illustrate that the scaled potentials $\gamma_n  e_{i, n}$ are  simply truncations of the Green kernel:
\begin{equation*}
    \gamma_n e_{i, n}(y) = G(X_i, y) \wedge \gamma_n
\end{equation*}

We emphasize that throughout this paper, we notably do not assume the smoothness of the underlying domain $\mathcal{X}$. Instead, we examine the interaction of the analytic (local Poincare inequality), geometric (volume doubling), and probabilistic (exit time bounds) properties of the Dirichlet form in determining the consistency of the sample estimator $\hat{f}_{D, \lambda}$ (see section \ref{Assumptions} and remarks therein). Hence, our analysis readily applies to both the smooth setting of Riemannian manifolds, and rough metric measure spaces, such as fractals, that do not possess geodesics. 

\subsection{Related Work}

Nonparametric regression in the first-order Euclidean Sobolev spaces $\mathbb{H}^1(\mathbb{R}^d)$ has been considered in a series of recent works \cite{green2021minimax, green2023minimax, garcia2020maximum, trillos2022rates}. Their techniques hinge on a careful spectral analysis of a smoothed graph Laplacian, which is harnessed either via regularization \cite{green2021minimax, trillos2022rates} or principal components regression \cite{green2023minimax}. In \cite{green2023minimax}, the authors observe that graph-based principal components regression, like its population counterpart spectral series regression \cite{lee2016spectral, dhillon2013risk}, enjoys superior consistency guarantees relative to Laplacian regularization. While graph-based approaches enjoy computational tractability, the latter works focus on estimating the \textit{in-sample} mean-square error. In practice, out-of-sample error is of greater interest, with risk typically measured with respect to the population. \cite{green2021statistical} briefly considers out-of-sample extensions with kernel smoothers, however their comparison with the in-sample MSE relies on strong smoothness assumptions on the underlying sampling measure, which is typically unknown. Moreover, the kernel smoothing technique is only applicable on Euclidean space, and has no tractable counterpart on general Dirichlet metric measure spaces which are considered here. \\

Nonparametric estimation on more general metric spaces has been studied in \cite{castillo2014thomas, efromovich2000sharp, angers2005multivariate, griebel2018regularized}, with a primary focus on compact manifolds. Like the kernel regression literature, these works only treat the supercritical regime where the hypothesis class contains bounded, continuous functions. Moreover, these works employ the machinery of heat-kernel generated frames that crucially involve the eigenfunctions of the population covariance operator, which is typically unknown. \\

We note that Dirichlet spaces arise as a natural hypothesis class in several machine learning tasks, due to their inherent association with an energy functional (the Dirichlet form) which can concisely capture the intrinsic geometry of both point clouds and their continuum limits. Indeed, on graphs, Dirichlet spaces have classically served as the canonical smoothness class \cite{dong2016learning} for graph signals, where the graph Dirichlet norm is understood to reflect the ``fitness'' between the signal and the graph topology \cite{hu2013graph}. Over the past two decades, a wealth of works \cite{singer2006graph, coifman2008graph, belkin2006convergence} have explored the convergence of this discrete notion of smoothness to its continuous counterpart in the large data limit, typically with the objective of quantifying the error incurred in approximating some population-level statistic by a corresponding graph-based sample estimator. While this convergence phenomenon is well-understood on Riemannian manifolds with quasi-uniformly sampled data \cite{belkin2006convergence}, this question remains highly nontrivial in rough non-geodesic metric measure spaces, such as fractals, where classical renormalizations break down \cite{akwei2024convergence}. Even when convergence is guaranteed, estimators constructed using Laplacian regularization typically possess suboptimal sample complexity \cite{green2023minimax} due to the approximation error between the graph and manifold ridge penalties. Here, we will demonstrate that we can bypass these concerns by working directly with the continuous-space Dirichlet energy on a larger, known ambient domain containing the support of the sampling measure. In certain cases, our continuous-space Dirichlet penalty can be viewed, in a sense, as an approximate continuum limit of a smoothed graph Laplacian constructed with (sub)-Gaussian edge weights (see Remark \ref{Continuum Limit})
\subsection{Contributions}
In this paper, we study the \textit{out-of-sample} (population) risk of nonparametric regressors in \textit{subcritical} Dirichlet spaces, where the learning target may not be continuous in the ambient topology. In our main result, Theorem \ref{Main Theorem}, we obtain the first \textit{optimal} out-of-sample convergence guarantees for nonparametric regression in general subcritical Dirichlet spaces. In contrast to previous works, our analysis considers general metric measure spaces, including both smooth domains, such as geodesic Riemannian manifolds, as well as nonsmooth domains such as graphs and fractals. As previewed in section \ref{Random Obstacle Preview}, our approach is based on directly renormalizing the continuous space Green's function, which we demonstrate is equivalent to smoothing the elements of the Dirichlet space by certain ``local averages'' in order to provide pointwise definition. Our approach is inspired by the construction of Liouville quantum gravity for log-correlated Gaussian fields \cite{duplantier2011liouville, duplantier2017log}, where the random field is similarly renormalized via ``circle averages''. In fact, the capacitary mean in \eqref{eq: Capacitary Mean} is a direct generalization of circle averaging to abstract Dirichlet spaces (see section \ref{Capacity Preliminaries}). A cornerstone of our approach is the independence of our renormalization technique from the sampling measure, which follows from the invariance of the Green's function under time-change. We strive to make minimal assumptions on the sampling measure itself, requiring only a doubling condition and a scale-invariant Poincar\'e inequality (see section \ref{Assumptions}). \\

Our work also contributes to the broader literature on nonparametric estimation over non-pregaussian classes. Indeed, subcritical Dirichlet spaces are non-pregaussian (and thereby non-Donsker), and namely do not admit well-defined notions of uniform or bracketing metric entropies due to their elements lacking pointwise definition. Hence, classical empirical process  techniques \cite{gine2016mathematical, van1996weak} are inapplicable, motivating the integral operator approach adopted in this work. Nevertheless, in the setting of ridge regression, our random obstacle estimator achieves optimal convergence rates, which are typically unattainable in the non-Donsker regime. For pure empirical risk minimization (constrained least squares), our estimator exhibits the same consistency guarantee expected for supercritical non-Donsker classes with the $L^2(\mathbb{P})$ covering metric entropy playing the role of the $L^2(\mathbb{P})$ bracketing entropy (see discussion after Theorem \ref{ERM Risk}). 

\subsection{Overview of the Paper}
In section \ref{Problem Setup} we introduce the data-generating model, review the basic theory of Dirichlet forms (section \ref{Preliminaries}), and outline our key assumptions (section \ref{Assumptions}). In section \ref{Random Obstacle Regression}, we review the formulation of random obstacle regression (section \ref{Green Truncation}), and present our main consistency result (section \ref{Main Consistency Result}) which demonstrates that renormalized ridge regression is rate-optimal. In section \ref{Random Obstacle ERM}, we consider pure empirical risk minimization over Dirichlet balls, and demonstrate that under stronger assumptions on the curvature of the metric measure space and the additive noise in our data-generating model, we can improve dependence on the confidence level at the expense of optimal sample complexity in our upper bound. 

\section{Problem Setup}
\label{Problem Setup}
Let $(\mathcal{M}, \mu, d)$ be a locally compact, separable, metric measure space and $\mathcal{X} \subset \mathcal{M}$ a compact, proper subset. We consider a probability measure $\nu$ obtained from the restriction of the ambient measure to $\text{supp}(\nu) = \mathcal{X}$, i.e. 
\begin{equation}
\label{eq: Sampling Measure}
    \nu = \frac{\mu|_{\mathcal{X}}}{\mu(\mathcal{X})}
\end{equation}
We suppose our data $D = \{(X_1, Y_1), \ldots, (X_n, Y_n)\}$ is sampled i.i.d, $(X, Y) \sim \mathbb{P}$, from an additive noise model:
\begin{equation}
\label{eq: Additive Noise Model}
    Y = f^{*}(X) + \epsilon
\end{equation}
where $\mathbb{P}|_{\mathcal{X}} = \nu$, $X \indep \epsilon$, $\mathbb{E}[\epsilon] = 0$, $\mathbb{E}[\epsilon^2] = \rho^2$ and $f^{*} \in \mathcal{H} \cap L^2(\nu)$ (defined below). In order for \eqref{eq: Additive Noise Model} to be meaningful, we understand $f^{*}$ to be a fixed representative in its $L^2(\nu)$-equivalence class, which we may always assume is quasi-continuous (see section \ref{Capacity Preliminaries}). \\

Throughout this paper, we assume \textbf{$\mathcal{X}^{\circ} \subset \mathcal{M}$ is an $A$-uniform domain} (here $^{\circ}$ denotes the topological interior). An open subset $S \subset \mathcal{M}$ is called $A$-uniform \cite{murugan2024heat} for $A > 1$ if for every $x, y \in S$, there exists a curve $\gamma: [0, 1] \to \mathcal{M}$ with $\gamma(0) = x$ and $\gamma(1)=y$ such that $\text{diam}(\gamma) \leq Ad(x, y)$ and for each $t \in (0, 1)$:
\begin{equation*}
    d(\gamma(t), \mathcal{M} \setminus S) \geq A^{-1}\min \{d(x, \gamma(t)), d(y, \gamma(t))\}
\end{equation*}

Uniform domains are abundant in various contexts including the Sobolev extension property \cite{jones1981quasiconformal}, Gromov hyperbolicity \cite{bonk2001uniformizing}, and geometric function theory \cite{martio1979injectivity}, and include the more commonly studied Lipschitz and NTA domains studied in the Euclidean setting. It is important to note uniform domains may be readily constructed in both smooth settings, such as length spaces, as well as rough settings, such as fractals that do not possess geodesics \cite{rajala2021approximation, murugan2024heat}.

\subsection{Dirichlet Form Preliminaries}
\label{Preliminaries}
We consider a Dirichlet form $\mathcal{E}$ on $L^2(\mathcal{M}, \mu)$ giving us a metric measure Dirichlet space (MMD) $(\mathcal{M}, \mu, d, \mathcal{E})$.  The domain of $\mathcal{E}$, denoted by $\mathcal{H}$, is assumed to be dense in $L^2(\mathcal{M}, \mu)$ and will be referred to as a Dirichlet space. We suppose $(\mathcal{E}, \mathcal{H})$ is \textit{regular}, which means $C_0(\mathcal{M}) \cap \mathcal{H}$ is dense in both $\mathcal{C}_0(\mathcal{M})$ and $\mathcal{H}$, where $\mathcal{C}_0(\mathcal{M})$ is the space of continuous functions with compact support in $\mathcal{M}$ (with uniform norm) and $\mathcal{H}$ is endowed with norm $||\cdot||_{\mathcal{H}_1}$:
\begin{equation*}
    ||g||_{\mathcal{H}_1} \equiv \sqrt{\mathcal{E}(g, g) + ||g||^2_{L^2(\mu)}}
\end{equation*}
Without loss of generality, we always identify $g \in \mathcal{H}$ with its quasi-continuous modification (which exists by regularity). We suppose our form is \textit{strongly local}, which means that $\mathcal{E}(g, h) = 0$ for all $g, h \in \mathcal{H}$ with compact supports such that $g$ is constant in a neighborhood of the support of $h$. \\

Let $\Delta$ denote the generator of $(\mathcal{E}, \mathcal{H})$, a self-adjoint operator on $L^2(\mathcal{M}, \mu)$ with domain dense in $\mathcal{H}$ such that for all $h \in \text{dom}(\Delta)$ and $g \in \mathcal{F}$, we have:
\begin{equation*}
    \mathcal{E}(h, g) = -\int (\Delta h)g d\mu
\end{equation*}
The associated heat semigroup:
\begin{equation*}
    P_t = e^{t\Delta}
\end{equation*}
is a family of contractive, strongly continuous, self-adjoint operators on $L^2$ that are \textit{Markovian}. In addition, a family $\{p_t\}_{t \geq 0}$ of nonnegative, $\mu \times \mu$ measurable functions on $\mathcal{M} \times \mathcal{M}$ is called the \textit{heat kernel} of $(\mathcal{E}, \mathcal{H})$ if they serve as integral kernels for the semigroup $\{P_t\}_{t \geq 0}$:
\begin{equation*}
    P_t h(x) = \int_{\mathcal{M}} p_t(x, y)h(y)d\mu(y)
\end{equation*}
for almost all $x \in \mathcal{M}$ and $h \in L^2(\mu)$. \\

Central to our analysis will be the \textit{Green function} of $(\mathcal{E}, \mathcal{H})$ defined $\mu \times \mu$ a.e. as:
\begin{equation}
\label{eq: Green Definition}
    G(x, y) = \int_{0}^{\infty} p_t(x, y)dt
\end{equation}
The existence of a Green's function is guaranteed under mild conditions (Lemma 5.2 in \cite{grigor2014heat}) which will always be satisfied here (we notably assume $(\mathcal{E}, \mathcal{H})$ is transient). Note, that while the Dirichlet space $(\mathcal{E}, \mathcal{H})$ is constructed over $L^2(\mathcal{M}, \mu)$, the Green kernel itself does not depend on $\mu$ and is invariant under time-change (change of reference measure; see section 6.2.1 in \cite{fukushima2010dirichlet}). Most importantly, if we define the operator $\mathcal{G}: L^2(\mathcal{M}, \mu) \to \mathcal{H}$ a.e. by:
\begin{equation*}
    \mathcal{G}h(x) = \int G(x, y)h(y)d\mu(y)
\end{equation*}
for $h \in L^2(\mathcal{M}, \mu)$, then $\mathcal{G}$ ``inverts'' the generator $-\Delta$:
\begin{equation*}
    \mathcal{E}(\mathcal{G}h, g) = \int_{\mathcal{M}} h(x)g(x)d\mu(x)
\end{equation*}
for all $g \in \mathcal{H}$. In this way, $G(x, y)$ resembles a reproducing kernel --- however it is crucially \textit{not reproducing} because $G(x, \cdot) \not\in L^2(\mathcal{M}, \mu)$ for any $x \in \mathcal{M}$ as it possesses a singularity at $x$ (in particular  $G(x, \cdot) \not\in \mathcal{H}$). A central focus of this paper is on renormalizing the Green function $G$ so that it may represent a continuous linear functional on $\mathcal{H}$. \\

By its Markovian property, the semigroup $\{P_t\}_{t \geq 0}$ additionally generates a \textit{Hunt process} $\{X_t\}_{t \geq 0}$ with infintesimal generator $\Delta$. A Dirichlet form is \textit{transient} iff $\{X_t\}_{t \geq 0}$ is a transient stochastic process. \textbf{In this paper, we always assume the Dirichlet form $(\mathcal{E}, \mathcal{H})$ is transient}. For more details on the interaction of the probabilistic and analytic properties of $(\mathcal{E}, \mathcal{H})$, please refer to \cite{fukushima2010dirichlet}. 

\begin{example}
As mentioned previously, the canonical example of a Dirichlet space is the first order Sobolev space $\mathbb{H}_0^1(\mathbb{R}^d, dx)$ with Dirichlet boundary conditions. The Dirichlet form for this space is given by:
\begin{equation}
\label{eq: Riemannian Gradient}
    \mathcal{E}(g, h) = \int_{\mathbb{R}^d} (\nabla g \cdot \nabla h)(x) dx
\end{equation}
Other important examples of Dirichlet spaces are formed on Riemannian manifolds, homogeneous spaces, and fractals, see \cite{barlow2006diffusions, kigami2001analysis}. 
\end{example}
Note that for each $g \in \mathcal{H}$ we may define the \textit{energy measure} $\Gamma(g, g)$ associated with $g$ as the unique positive Radon measure such that for any $u \in \mathcal{H} \cap C_0(\mathcal{M})$:
\begin{equation*}
    \int_{\mathcal{M}} u \hspace{0.5mm} d\Gamma(g, g) = \mathcal{E}(ug, g) - \frac{1}{2}\mathcal{E}(u, g^2)
\end{equation*}

For $g, h \in \mathcal{H}$, we likewise introduced the signed measure $\Gamma(g, h)$ defined as:
\begin{equation*}
    \Gamma(g, h) = \frac{1}{2}\Big(\Gamma(h + g, h + g) - \Gamma(h, h) - \Gamma(g, g)\Big)
\end{equation*}
so that
\begin{equation*}
    \int_{\mathcal{M}} u \hspace{0.5mm} d\Gamma(g, h) =  \frac{1}{2}\Big(\mathcal{E}(ug, h) + \mathcal{E}(g, uh) - \mathcal{E}(u, gh)\Big)
\end{equation*}
Clearly:
\begin{equation*}
    \int_{\mathcal{M}} d\Gamma(g, h) = \mathcal{E}(g, h)
\end{equation*}

The energy measure enjoys the chain, Leibniz, and Cauchy-Schwarz properties \cite{grigor2015generalized} --- namely  for all $u, w, \phi, \psi \in \mathcal{F} \cap C_0(\mathcal{M})$.

\begin{align}
    d\Gamma(\phi, uw) & = u \hspace{0.5mm} d\Gamma(\phi, w) + w \hspace{0.5mm} d\Gamma(\phi, u) \label{eq: Leibniz Def} \\ 
    d\Gamma(\Psi(u), w) & = \Psi'(u) d\Gamma(u, w) \label{eq: Chain Def} \\
    \Big|\int \phi \psi \hspace{0.5mm} d\Gamma(u, w)\Big| & \leq \sqrt{\int \phi^2 d\Gamma(u, u) \int \psi^2 d\Gamma(w, w)} \label{eq: CS Inequality}
\end{align}

\subsubsection{Capacity}
\label{Capacity Preliminaries}

Denote by $\mathcal{O}$ the set of open subsets of $\mathcal{M}$. For any open set $A \in \mathcal{O}$, we define:
\begin{equation*}
    \mathcal{L}(\mathcal{A}) = \{g \in \mathcal{H}_{e}: g \geq \mathbb{1}_{\mathcal{A}} \hspace{2mm} \mu-\text{a.e.}\}
\end{equation*}
where $\mathcal{H}_e$ is the extended Dirichlet space (see section \ref{Trace Preliminaries} below). Then we define the $\mathcal{E}$-capacity  $\text{Cap}(\mathcal{A})$ as:
\begin{equation}
\label{eq: Capacity Definition}
    \text{Cap}(\mathcal{A}) = \inf_{g \in \mathcal{L}(\mathcal{A})} \mathcal{E}(g, g)
\end{equation}
By the transience of $(\mathcal{E}, \mathcal{H})$, there is a unique minimizer $e_{\mathcal{A}} \in \mathcal{L}_{\mathcal{A}}$ achieving the infimum in \eqref{eq: Capacity Definition}, which we call the \textit{equilibrium potential} of $\mathcal{A}$ (the existence of this minimizer is guaranteed by the convexity of $\mathcal{L}(\mathcal{A})$ and its closure with respect to $\mathcal{E}$; see section 2.1 of \cite{fukushima2010dirichlet}). Associated with $e_{\mathcal{A}}$ is a unique measure $\nu_{A}$ supported on $\bar{\mathcal{A}}$ called the \textit{equilibrium measure of} $\mathcal{A}$ which represents the action of $e_{A}$, i.e.:

\begin{equation*}
    \mathcal{E}(h, e_{A}) = \int_{\mathcal{M}} \tilde{h}(x) d\nu_{A}(x)
\end{equation*}

where $\tilde{h}$ denotes a quasi-continuous representative of $h \in \mathcal{H}$. We say a property holds \textit{quasi-everywhere} (denoted q.e.) if it holds everywhere in $\mathcal{M}$ except for a set of capacity zero. By the regularity of $(\mathcal{E}, \mathcal{H})$, every $h \in \mathcal{H}$ possesses a quasi-continuous representative $\tilde{h}$--- in the sequel we always identify with this representative ($h = \tilde{h}$).


\subsubsection{Extended Dirichlet Space}
\label{Trace Preliminaries}
Note that $\sqrt{\mathcal{E}}$ is not organically a norm on $\mathcal{H}$ (indeed observe we discussed completion with respect to $\|\cdot\|_{\mathcal{H}_1}$). When $(\mathcal{E}, \mathcal{H})$ is transient (as we will always assume here), we may extend $\mathcal{H}$ to produce the \textit{extended Dirichlet space} $\mathcal{H}_{e}$:
\begin{equation*}
    \mathcal{H}_e \equiv \{g: \exists \{g_n\}_{n = 1}^{\infty} \subset \mathcal{H} \hspace{1mm} \text{is} \hspace{1mm} \text{$\mathcal{E}$-Cauchy}, g = \lim_{n} g_n \hspace{2mm} \mu-a.e.\}
\end{equation*}

By Theorem 1.5.2 in \cite{fukushima2010dirichlet}, $\mathcal{H}_e$ is a Hilbert space with inner product $\mathcal{E}(\cdot, \cdot)$ and norm $\sqrt{\mathcal{E}}$. Here, we adopt the convention of Beurling-Deny and further assume $\mathcal{H}_{e}$ is regular and $\mathcal{H}_{e} \subset L^1_{\text{loc}}(\mathcal{X}, \mu)$. \\

Now, suppose $m$ is a smooth measure (see \cite{fukushima2010dirichlet} for a definition) with compact support such that $\mathcal{H}_e \subset L^1(\mathcal{M}, m)$. Then, we have, by regularity, that $h \cdot m$ is a finite Radon measure for every $h \in \mathcal{H}^{+}_e$, hence, by Theorem 2.2.4(ii) of \cite{fukushima2010dirichlet}, $h \cdot m$ is of finite $0$-order energy integral, and therefore there exists a potential $U(h \cdot m) \in \mathcal{H}_e$ and operator $C_{m}: \mathcal{H}_e \to \mathcal{H}_e$ such that:
\begin{equation}
\label{eq: Covariance Reproducing}
    \mathcal{E}(C_{m}h, g) = \mathcal{E}(U(h \cdot m), g) = \int_{\mathcal{M}} h(y)g(y) dm(y)
\end{equation}
for all $g \in \mathcal{H}_e$. It is easy to see that $C_{m}$ is self-adjoint and positive definite. Moreover, a quasi-continuous version of $C_{m}h$ is given by:
\begin{equation}
\label{eq: 0-order resolvent}
    C_{m}h(x) = \int_{\mathcal{M}} G(x, y)h(y)dm(y)
\end{equation}

Note that our regression function $f^{*} \in L^2(\nu)$ by definition. Observe that $\mathcal{H}_e$ admits the orthogonal decomposition:
\begin{equation}
\label{Decomposition of Extended Dirichlet Space}
    \mathcal{H}_e = \mathcal{H}_{\mathcal{M} \setminus \mathcal{X}} \otimes \mathscr{H}_{\mathcal{X}}
\end{equation}
where $\mathcal{H}_{\mathcal{M} \setminus \mathcal{X}} = \{g \in \mathcal{H}_e: g = 0 \hspace{1mm} \text{on} \hspace{1mm} \mathcal{X} \hspace{2mm} \text{q.e.}\}$. Abusing notation, we always identify $f^{*}$ with its harmonic extension in $\mathcal{M} \setminus \mathcal{X}$, i.e. $P_{\mathcal{X}}\tilde{f}^{*} \in \mathscr{H}_{\mathcal{X}}$ where $P_{\mathcal{X}}$ is the orthogonal projection onto $\mathscr{H}_{\mathcal{X}}$ and $\tilde{f}^{*} \in \mathcal{H}_e$ such that $\tilde{f}^{*} = f^{*}$ $\nu$-a.e.
\begin{remark}[Abuse of Notation]
To ease readability, in this paper we abuse notation and \textbf{denote the extended Dirichlet space also as $\mathcal{H}$} (noting that our regression function $f^{*} \in \mathcal{H} \cap L^2(\nu)$).  
\end{remark}
\subsection{Assumptions}
\label{Assumptions}
We impose the following assumptions on our Dirichlet form and metric measure space:

\subsubsection{Volume Doubling}
We say $\mu$ is volume doubling \eqref{eq: VD} if for any ball $B = B(x, r) \subset \mathcal{M}$
\begin{equation}
\label{eq: VD}
\tag{VD}
    \mu(2B) \leq C\mu(B) 
\end{equation}
where $C > 1$ is independent of the ball $B$. We define the volume exponent:
\begin{equation*}
    \alpha = \frac{\log C}{\log 2}
\end{equation*}
We say that $\mu$ is reverse volume doubling, if for any ball $B = B(x, r) \subset \mathcal{M}$:
\begin{equation}
\label{eq: RVD}
\tag{RVD}
    \mu(B) \leq K^{-1}\mu(2B)
\end{equation}
where again the constant $K > 1$ is independent of the ball $B$. Define the exponent:
\begin{equation*}
    \alpha' = \frac{\log K}{\log 2}
\end{equation*}
Note that it is known that \eqref{eq: VD} implies \eqref{eq: RVD} when $\mathcal{M}$ is unbounded and connected.  \\

For a ball $B(x, r)$, we denote its volume by:
\begin{equation*}
    V(x, r) = \mu(B(x, r))
\end{equation*}

We will further assume that there exists $\tilde{r} \in (0, 1)$, so that for all $r \in (0, \tilde{r})$:
\begin{equation}
\label{eq: VD/RVD}
\tag{V}
    \tilde{C}^{-1} r^{\alpha} \leq V(x, r) \leq \tilde{C} r^{\alpha}
\end{equation}
for all $x \in \mathcal{M}$ and some $\tilde{C} > 1$ independent of $x$. 

\subsubsection{Poincar\'e Inequality}
We say the \textit{Poincar\'e inequality} \eqref{eq: Poincare Inequality} holds if there exists constants $C_1 > 0$, $\sigma \in (0, 1)$  and a nondecreasing function $\Psi(\cdot)$ such that for all balls $B = B(x, r) \subset \mathcal{M}$ and $h \in \mathcal{H}$:

\begin{equation}
\label{eq: Poincare Inequality}
\tag{PI}
    \int_{\sigma B} (h - h_{\sigma B})^2 d\mu \leq C_1\Psi(r)\int_{B} d\Gamma(h, h)
\end{equation}
where for $\mathcal{A} \subset \mathcal{M}$, $h_{A} = \frac{1}{\mu(A)} \int_{A} h(x) d \mu(x)$ is the mean, and $d\Gamma(h, h)$ is the energy measure of $h$, with respect to the Dirichlet form $(\mathcal{E}, \mathcal{H})$. 

\subsubsection{Mean Exit Time Bounds}
We say the \textit{mean exit time bounds} \eqref{eq: Exit Upper}/\eqref{eq: Exit Lower} hold if there exists a constant $C_2 > 1$ such that for all balls $B = B(x, r) \subset \mathcal{M}$:

\begin{align}
    \sup_{x \in B} \mathbb{E}_x[\tau_{B}] & \leq C_2\Psi(r)  \label{eq: Exit Upper} \tag{UE}\\
    \inf_{x \in B} \mathbb{E}_x[\tau_{B}] & \geq C_2^{-1}\Psi(r) \label{eq: Exit Lower} \tag{LE}
\end{align}
where $\tau_B$ denotes the exit time of $\{X_t\}_{t \geq 1}$ (the Hunt process associated with $\mathcal{E}$) from the ball $B$. \\

Throughout this paper, we will impose the following growth condition on $\Psi$. Suppose for all $0 < r_1 < r_2$:
\begin{equation}
\label{eq: Growth Condition}
\tag{$\Psi$}
    \frac{1}{C_{\Psi}}\Big(\frac{r_2}{r_1}\Big)^{\beta} \leq \frac{\Psi(r_2)}{\Psi(r_1)} \leq C_{\Psi} \Big(\frac{r_2}{r_1}\Big)^{\beta'}
\end{equation}
for some $C_{\Psi} \geq 1$ and $1 < \beta \leq \beta'$, with $\beta' < \alpha'$ and $\beta < \alpha$ (recall $\alpha$ and $\alpha'$ are the volume doubling and reverse volume doubling exponents, respectively). Moreover, we assume there exists a $r_0 \in (0, 1)$ such that for all $0 < r \leq r_0 \leq 1$:
\begin{equation*}
\tag{$\Psi$}
    \tilde{C}^{-1}_{\Psi}r^{\beta} \leq \Psi(r) \leq \tilde{C}_{\Psi} r^{\beta}
\end{equation*}
for some $\tilde{C}_{\Psi} > 1$. 
\subsubsection{Remarks on Assumptions}
\label{Assumption Remarks}
Before proceeding, we make some remarks on the common space-time scaling function $\Psi(r)$ in the above assumptions, and emphasize the generality of our framework. We first note that on any geodesically complete Riemannian manifold (with nonnegative curvature), \eqref{eq: Poincare Inequality} and \eqref{eq: Exit Upper}/\eqref{eq: Exit Lower} are satisfied with $\Psi(r) = r^2$ by the canonical Dirichlet form \eqref{eq: Riemannian Gradient}. More generally, the scale-invariant Poincare inequality \eqref{eq: Poincare Inequality} with $\Psi(r) = r^2$ is satisfied on $\text{RCD}^{*}$ metric measure spaces (with curvature bounded below), Carnot-Caratheodory spaces, domains of uniformly elliptic and even some degenerate (sub)-elliptic/parabolic operators (see further examples in \cite{coulhon2012heat, coulhon2020gradient}). In the fractal setting, the Sierpinski carpet on $\mathbb{R}^d$ with $d \geq 3$ possesses a canonical, transient Dirichlet form, and satisfies \eqref{eq: Poincare Inequality}, \eqref{eq: Exit Upper}/\eqref{eq: Exit Lower}, and \eqref{eq: Growth Condition} with $\beta$ equal to the walk dimension of the associated Brownian motion \cite{barlow2013analysis}. More generally, the combination of \eqref{eq: Poincare Inequality} and \eqref{eq: VD} implies \eqref{eq: Exit Upper}/\eqref{eq: Exit Lower} in the presence of upper and lower heat kernel estimates \cite{grigor2015generalized}. Clearly, \eqref{eq: VD} and \eqref{eq: RVD} are satisfied by the Lebesgue measure on the Euclidean Sobolev space $\mathbb{H}^1_{0}(\mathbb{R}^d)$, moreover on many fractal spaces we obtain $V(x, r) \asymp r^{\alpha}$ for all $x \in \mathcal{M}$ and $r > 0$ \Big(e.g. $\alpha = \frac{\log 3^{d} - 1}{\log 3}$ on the Sierpinski carpet on $\mathbb{R}^d$\Big). Moreover, \eqref{eq: VD}, \eqref{eq: RVD}, \eqref{eq: VD/RVD} are satisfied whenever $\mu$ admits a density, bounded above and below, with respect to the $\alpha$-dimensional Hausdorff measure. \\

Crucially, by Theorem 1.2 in \cite{grigor2015generalized}, the combination of \eqref{eq: Poincare Inequality}, \eqref{eq: Exit Upper}, and \eqref{eq: Exit Lower} imply the following \textit{cutoff Sobolev inequality}: \newline

\textbf{Cutoff Sobolev inequality} \newline
A cutoff function $\phi \in \mathcal{H}$ in $\text{cutoff}(\mathcal{A}, \mathcal{B})$ satisfies:
\begin{itemize}
    \item $0 \leq \phi \leq 1$ on $\mathcal{M}$
    \item $\phi \equiv 1$ on $\mathcal{A}$ 
    \item $\phi \equiv 0$ on $\mathcal{M} \setminus \mathcal{B}$
\end{itemize}
We say the cutoff Sobolev inequality \eqref{eq: Cutoff Sobolev} holds if there exists constants $C_3, C_4 > 0$ and a nondecreasing function $\Psi(\cdot)$ such that for any concentric balls $B_1 = B(x, R)$ and $B_2 = B(x, R + r)$, there exists a cutoff function $\phi \in \text{cutoff}(B_1, B_2)$ such that for any measurable function $h \in \mathcal{H} \cap L^{\infty}$ we have:
\begin{equation}
\label{eq: Cutoff Sobolev}
\tag{CSA}
    \int_{B_2 \setminus B_1} h^2 \hspace{0.5mm} d\Gamma(\phi, \phi) \leq C_3\int_{B_2} d\Gamma(h, h) + \frac{C_4}{\Psi(r)}\int_{B_2 \setminus B_1} h^2 d\mu
\end{equation}

In \cite{murugan2020length}, it was demonstrated (Corollary 1.10) that when $(\mathcal{M}, d)$ satisfies \eqref{eq: Poincare Inequality}, \eqref{eq: VD}, and \eqref{eq: Cutoff Sobolev} then:
\begin{equation*}
    \Psi(r) \preceq r^2
\end{equation*}
as $r \to 0$ (i.e. the walk dimension is at least 2). Moreover, \cite{kajino2020singularity} showed that if \begin{equation*}
    \lim \inf_{\lambda \to \infty, r \to 0} \frac{\lambda^2 \Psi(r \lambda^{-1})}{\Psi(r)} = 0
\end{equation*}
then $d\Gamma(g, g)$ is singular with respect to $\mu$ for every $g \in \mathcal{H}$ (i.e. the space is fractal). Indeed, fractals are characterized by their heat kernels enjoying \textit{subgaussian} bounds, as opposed to the Gaussian bounds observed on geodesic Riemannian manifolds.

\subsubsection{$\mathcal{E}$ as a continuum limit}
\label{Continuum Limit}
In certain smooth spaces $\mathcal{M}$, $\mathcal{E}$ can roughly be viewed as the continuum limit of graph Dirichlet energies. Note, by Theorem 1.2 in \cite{grigor2015generalized}, we have that the combination of \eqref{eq: VD}, \eqref{eq: Poincare Inequality}, and \eqref{eq: Exit Upper}/\eqref{eq: Exit Lower} implies the following heat kernel estimates:
\begin{align}
    p_t(x, y) & \leq \frac{C}{V(x, t^{\frac{1}{\beta}})}\text{exp}\Big(-c\Big(\frac{d(x, y)^{\beta}}{t}\Big)^{\frac{1}{\beta - 1}}\Big) \label{eq: Upper Heat Kernel} \\
    p_t(x, y) & \geq \frac{C^{-1}}{V(x, t^{\frac{1}{\beta}})}\text{exp}\Big(-c\Big(\frac{d(x, y)^{\beta}}{t}\Big)^{\frac{1}{\beta - 1}}\Big) \hspace{2mm} \text{for} \hspace{1mm} d(x, y)^{\beta} \leq \epsilon t \label{eq: Lower Heat Kernel}
\end{align}
for some absolute constants $C > 1$ and $c, \epsilon > 0$. Now, suppose $\mathcal{M}$ is compact. Recall the definition of $\mathcal{E}$:
\begin{equation}
\label{eq: Semigroup Approximation}
    \mathcal{E}(f, f) = \lim_{t \to 0}  \mathcal{E}_{t}(f, f) \equiv \lim_{t \to 0} \frac{1}{t}\int_{\mathcal{M}} \int_{\mathcal{M}} p_t(x, y)(f(x) - f(y))^2 d\mu(x)d\mu(y)
\end{equation}
Let $\mathcal{E}_{t, n}(f, f)$ denote the sample form obtained from $\mathcal{E}_{t}(f, f)$ by replacing $\mu$ in \eqref{eq: Semigroup Approximation} with its empirical counterpart $\mu_n$. Then, $\mathcal{E}_{t, n}(f, f)$ can be roughly viewed as a graph approximation to $\mathcal{E}(f, f)$, constructed using a smoothed Laplacian with kernel:
\begin{equation*}
    K(x, y) = \frac{1}{V(x, t^{\frac{1}{\beta}})}\text{exp}\Big(-c\Big(\frac{d(x, y)}{t^{\beta}}\Big)^{\frac{1}{\beta - 1}}\Big) 
\end{equation*}
Note, this is merely a heuristic comparison, due to the local nature of the lower estimate in \eqref{eq: Lower Heat Kernel} and the ambiguity of the constants in \eqref{eq: Upper Heat Kernel}/\eqref{eq: Lower Heat Kernel}, there is no guarantee that such a smoothed graph Laplacian would converge to its continuous-space equivalent $\Delta$ in the large data limit. We refer the reader to  \cite{hinz2016closability} for more details regarding graph approximations to MMD spaces. 


\begin{remark}[Notation]
We use the notation $A \preceq B$ (likewise $A \succeq B$) when there exists an implicit constant $C > 0$ (dependent only on inessential parameters) such that $A \leq CB$ (likewise $A \geq CB$). When $A \preceq B$ and $A \succeq B$ we write $A \asymp B$. We will also repeatedly invoke two (semi-)norms:
\begin{align}
    ||h||^2_{\mathcal{H}} &  = \mathcal{E}(h,  h) \label{eq: 0-norm} \\
    ||h||^2_{\mathcal{H}_1} &  = \mathcal{E}(h,  h) +  \|h\|^2_{L^2(\nu)} \label{eq: 1-norm} 
\end{align}
for all $h \in \mathcal{H}$. For Banach spaces $\mathcal{A}$, $\mathcal{B}$ and an operator $T: \mathcal{A} \to \mathcal{B}$, we express its norm as $\|T\|_{\mathcal{A} \to \mathcal{B}}$. Further, we write $\mathbb{B}(A)$ to denote the unit ball in $\mathcal{A}$ (centered at $0 \in \mathcal{A}$) and $\mathbb{B}(f, \mathcal{A})$ to denote a unit ball in $\mathcal{A}$ centered at $f \in \mathcal{A}$. 
\end{remark}
\section{Random Obstacle Regression}
\label{Random Obstacle Regression}
In this section, we study the consistency of $\hat{f}_{D, \lambda}$ defined in \eqref{eq: Capacitary Ridge Regression}. Before, we present our main consistency result in Theorem \ref{Main Theorem}, we first review and elaborate on our approach, which was previewed in section \ref{Random Obstacle Preview}. 
\subsection{Renormalizing the Green Function}
\label{Green Truncation}

Recall that we wish to renormalize the following ill-posed optimization problem:

\begin{equation}
\label{eq: Ill-Posed}
    \min_{g \in \mathcal{H}} \frac{1}{n}\sum_{i = 1}^n (Y_i - g(X_i))^2 + \lambda \mathcal{E}(g, g)
\end{equation}

by replacing pointwise evaluations $g(X_i)$ (which are not continuous on $\mathcal{H}$) with certain local averages of $g$ on the boundaries of obstacles $\mathcal{O}_{i, n}$. More generally, for any $x \in \mathcal{X}$, we define the obstacle $\mathcal{O}_{x, n}$ given by:
\begin{equation}
\label{eq: Obstacle Def}
    \mathcal{O}_{x, n}(\gamma_n) = \{y \in \mathcal{M}: G(x, y) \geq \gamma_n\}
\end{equation}
where $G$ is the Green function defined in \eqref{eq: Green Definition} and $\{\gamma_n\}_{n = 1}^{\infty} \uparrow \infty$ is some increasing sequence of thresholds  which will be chosen subsequently. Observe that as $\gamma_n \to \infty$, the obstacle $\mathcal{O}_{x, n}(\gamma_n)$ shrinks toward the point $x \in \mathcal{X}$. With each obstacle, we associate an equilibrium potential $e_{x, n}$ that solves: 
\begin{align}
e_{x, n} & = \text{arg} \min_{g \in \mathcal{L}(\mathcal{O}_{x, n})} \mathcal{E}(g, g)  \label{eq: Cap Minimizer} \\
\mathcal{L}(\mathcal{O}_{x, n}) & \equiv \{g \in \mathcal{H}: g \geq \mathbb{1}_{\mathcal{O}_{x, n}}\} \nonumber 
\end{align}
where the existence of a unique minimizer $e_{x, n} \in \mathcal{H}$ for \eqref{eq: Cap Minimizer} is ensured by the assumed transience of  $(\mathcal{E}, \mathcal{H})$. \\

Probabilistically, $e_{x, n}$ can be expressed a.e. as:
\begin{equation}
\label{eq: Harmonic Extension}
e_{x, n}(y) = P_y(T_{\mathcal{O}_{x, n}} < \infty)
\end{equation}
where $T_{\mathcal{O}_{x, n}}$ denotes the hitting time of the Hunt process $\{X_t\}_{t \geq 0}$ to the obstacle $\mathcal{O}_{x, n}$ (Theorem 4.3.3 in \cite{fukushima2010dirichlet}).  Then, it follows from the harmonicity of $G(x, \cdot)$ in $\mathcal{M} \setminus \{x\}$ that $e_{x, n}$ can be expressed as:
\begin{equation}
\label{Green Potential}
    e_{x, n}(\cdot) = \frac{G(x, \cdot) \wedge \gamma_n}{\gamma_n}
\end{equation}
In other words, $e_{x, n}$ is simply a scaled version of the truncated Green's kernel $G(x, \cdot)$ which has been cutoff at $\gamma_n$. Indeed, we have by harmonicity, \eqref{eq: Obstacle Def} and \eqref{eq: Harmonic Extension}, that for $y \not \in \mathcal{O}_{x, n}$:
\begin{align*}
    G(y, x) & = \tilde{\mathbb{E}}_y[G(X_{T_{\mathcal{O}_{x, n}}}, x)] \\
    & = \tilde{\mathbb{E}}_y[G(X_{T_{\mathcal{O}_{x, n}}}, x)\mathbb{1}_{T_{\mathcal{O}_{x, n}} < \infty}] +  0 \cdot P_y(T_{\mathcal{O}_{x, n}} = \infty) \\
    & = \gamma_n \cdot P_y(T_{\mathcal{O}_{x, n}} < \infty) \\
    & = \gamma_n e_{x, n}(y)
\end{align*}
where $\tilde{\mathbb{E}}$ denotes expectation with respect to the law of $\{X_t\}_{t \geq 0}$. Now, recall from section \ref{Random Obstacle Preview}, there is a unique \textit{equilibrium measure} $\nu_{x, n}$ concentrated on $\partial \mathcal{O}_{x, n}$ such that for all $h \in \mathcal{H}$:

\begin{equation*}
    \mathcal{E}(h, e_{x, n}) = \int h(y)d\nu_{x, n}(y)
\end{equation*}

and $\nu_{x, n}(\partial \mathcal{O}_{x, n}) = \gamma_n^{-1} = \text{cap}(\mathcal{O}_{x, n})$. Equipped with the equilibrium potential $e_{x, n}$ and measure $\nu_{x, n}$, we now define the mean functional:
\begin{equation*}
    P_{x, n} h = \mathcal{E}(h, \gamma_n e_{x, n}) = \frac{1}{\nu_{x, n}(\partial \mathcal{O}_{x, n})}\int h(y)d\nu_{x, n}(y)
\end{equation*}
When $x = X_i$ for some $i \in [n]$, we simply denote $e_{x, n}, \nu_{x, n}$, and $P_{x, n}$ by $e_{i, n}, \nu_{i, n}$, and $P_{i, n}$, respectively. For each $i \in [n]$, we now replace $g(X_i)$ with $P_{i, n}g$ in \eqref{eq: Ill-Posed} to obtain:
\begin{equation}
\label{eq: Well-Posed}
    \text{arg}\min_{g \in \mathcal{H}} \frac{1}{n}\sum_{i = 1}^n (Y_i - P_{i, n}g)^2 + \lambda\mathcal{E}(g, g)
\end{equation}
whose optimal solution $\hat{f}_{D, \lambda}$ may be expressed as:
\begin{equation}
\label{eq: Representer Theorem 2}
    \hat{f}_{D, \lambda} = \sum_{i = 1}^n c_i \gamma_n e_{i, n}
\end{equation}
where $\mathbf{c} = [c_1, \ldots, c_n]^T \in \mathbb{R}^n$ is given by:
\begin{align*}
    \mathbf{c} & = (\mathbf{G}_n + n\lambda \mathbf{I})^{-1}\mathbf{y} \\
    \mathbf{y} & = [Y_1, \ldots, Y_n]^T \\
    (\mathbf{G}_n)_{i, j} & = \gamma^2_n\mathcal{E}(e_{i, n}, e_{j, n})
\end{align*}
In light of \eqref{Green Potential}, \eqref{eq: Representer Theorem 2} may be expressed more naturally as:
\begin{equation}
\label{eq: Final Sample Estimator}
    \hat{f}_{D, \lambda}(\cdot) = \sum_{i = 1}^n c_i (G(X_i, \cdot) \wedge \gamma_n)
\end{equation}

We emphasize that the representers in \eqref{eq: Final Sample Estimator} depend only on the form $\mathcal{E}$, and \textit{require no knowledge of the underlying sampling measure $\nu$}. This is cornerstone of our approach --- indeed, while there are many possible ways to approximate the point mass at $X_i$ with a ``smoothed'' local average, we choose the particular obstacle in \eqref{eq: Obstacle} because the corresponding representers  in \eqref{eq: Representer Theorem 2} may be explicitly formulated with minimal knowledge of the geometry of the Dirichlet form. 

\begin{example}
\label{ex: Euclidean Sobolev example}
Let $\mathcal{H} = \mathbb{H}_0^1(\mathbb{R}^d)$. Then, $G(x, y) = \frac{\Gamma(\frac{d}{2} - 1)}{2\pi^{\frac{d}{2}}} \|x - y\|^{2-d}$. Hence, for any $x \in \mathbb{R}^d$ and $\gamma_n > 0$, we have (see eq. 2.3.55 in \cite{sznitman2013brownian}):
\begin{align*}
    O_{x, n} & = B\Big(x, \Big(\frac{\gamma_n}{c(d)}\Big)^{-\frac{1}{d-2}} \Big) \\
    e_{x, n}(y) & = \frac{c(d)\|x - y\|^{2 - d}}{\gamma_n} \wedge 1 \\
    d\nu_{x, n}(y) & = \gamma^{-1}_n dS_{x, R_n}(y) \hspace{2mm} \text{for } R_n = \Big(\frac{\gamma_n}{c(d)}\Big)^{-\frac{1}{d-2}} 
\end{align*}
where $dS_{x, R}$ denotes the uniform measure over the sphere $\partial B(x, R)$ and $c(d) = \frac{\Gamma(\frac{d}{2} - 1)}{2\pi^{\frac{d}{2}}}$. Observe that all three of $O_{x, n}, e_{x, n}$, and $\nu_{x, n}$ make no reference to the data-generating measure $\nu$. 
\end{example}
\subsection{Consistency of Random Obstacle Regression}
\label{Main Consistency Result}
We now present our main theorem on the consistency of $\hat{f}_{D, \lambda}$. 
\begin{theorem}
\label{Main Theorem}
Suppose \eqref{eq: VD}, \eqref{eq: RVD}, \eqref{eq: VD/RVD}, \eqref{eq: Poincare Inequality}, \eqref{eq: Exit Upper}, \eqref{eq: Exit Lower}, and \eqref{eq: Growth Condition} hold. Then, with a choice of $\gamma_n \asymp n^{\frac{\alpha - \beta}{\alpha + \beta}}$ and $\lambda \asymp n^{-\frac{\beta}{\alpha + \beta}}$, we obtain:
\begin{equation}
\label{eq: Main Theorem Statement}
    ||\hat{f}_{\lambda, D} - f^{*}||^2_{L^2(\nu)} \preceq \delta^{-1} n^{-\frac{\beta}{\alpha + \beta}} \|f^{*}\|^2_{\mathcal{H}}
\end{equation}
with probability $1 - 3\delta$
\end{theorem}

\begin{remark}
In the specific case of $\mathcal{H} = \mathbb{H}^1_0(\mathbb{R}^d)$, with $\alpha = d$ and $\beta = 2$, Theorem \ref{Main Theorem} yields:

\begin{equation*}
    ||\hat{f}_{\lambda, D} - f^{*}||^2_{L^2(\nu)} \preceq \delta^{-1} n^{-\frac{2}{d+2}}
\end{equation*}
which exhibits the known optimal rate of $n^{-\frac{2}{d + 2}}$. \\

We emphasize that the primary benefit of random obstacle renormalization over the more commonly studied spectral series regression stems from the renormalization procedure being wholly agnostic to the data-generating measure (due to the invariance of the Green's function under time-change). Indeed, spectral series regression involves approximating the Green's function by truncating its Mercer expansion, and thereby requires full knowledge of Laplacian eigenvalues/eigenfunctions, which crucially depend on the unknown data generating measure. The random obstacle renormalization technique introduced here approximates the Green's functions via a simple truncation (see \eqref{Green Potential}) and hence does not require any spectral information on the ambient Laplacian. Moreover, our renormalization technique also appears to avoid the computational costs of spectral series regression --- it is well-known that when the latter is able to achieve optimal rates, the Mercer series must contain at least $n^{\frac{\alpha}{\alpha + \beta}}$ terms. We suspect that regularizer $\lambda$ plays a crucial role in achieving the optimal convergence rate in \eqref{eq: Main Theorem Statement} --- in section \ref{Random Obstacle ERM}, we consider pure empirical risk minimization (ERM) over the ``smoothed'' Dirichlet ball under an additional curvature assumption, and argue, using empirical process techniques, that unlike ridge regression in Theorem \ref{Main Theorem}, pure ERM overfits due to the ``massiveness'' of $\mathcal{H}$, leading to suboptimal sample complexity. 

\end{remark}

In order to understand how $\hat{f}_{D, \lambda}$ approaches the true mean function $f^{*}$ asymptotically, it is clear we must understand how the capacitary means $P_{x, n}g$ approach pointwise evaluations $g(x)$ in $L^2(\nu)$ --- this is the topic of the next section.

\subsection{Poincar\'e Inequality for Capacitary Means}
\label{Capacitary Poincare Section}
In this section, we derive the main technical machinery for the analysis of the MSE $||\hat{f}_{\lambda, D} - f^{*}||^2_{L^2(\nu)}$. In the following result, we derive a global Poincar\'e-type inequality for the capacitary means $P_{x, n}g$. Proofs are provided in Appendix \ref{Poincare Proofs}.

\begin{theorem}
\label{Capacitary Poincare}
Suppose \eqref{eq: VD}, \eqref{eq: RVD}, \eqref{eq: VD/RVD}, \eqref{eq: Poincare Inequality}, \eqref{eq: Exit Upper}, \eqref{eq: Exit Lower}, and \eqref{eq: Growth Condition} hold. Then, for every $g \in \mathcal{H}$, we have:
\begin{equation}
\label{eq: Capacitary Poincare}
    \int_{\mathcal{X}} (g(x) - P^{\gamma}_x g)^2 d\nu(x) \preceq \gamma^{\frac{\beta}{\beta - \alpha}} \mathcal{E}(g, g)
\end{equation}
where 
\begin{equation}
\label{eq: Projection Definition}
    P^{\gamma}_x g = \mathcal{E}(g, \gamma e_x) = \frac{1}{\nu_x(\partial \mathcal{O}_x)}\int g(y) d\nu_{x}(y)
\end{equation} and $e_x$ is the equilibrium potential for the obstacle $$\mathcal{O}_x = \{y \in \mathcal{M}: G(x, y) \geq \gamma\}$$ centered at $x$ and $\nu_x$ the corresponding equilibrium measure. 
\end{theorem}

The reader may notice that, in light of \eqref{eq: Growth Condition}, Theorem \ref{Capacitary Poincare} resembles a global version of \eqref{eq: Poincare Inequality}, where local ball $\nu$-averages $g_{B}$ have been replaced by local capacitary means $P_x^{\gamma}g$. One might wonder why we don't just use a local ball $\nu$-average $g_{B(X_i, r_n)}$ instead of $P_{i, n}g$ in \eqref{eq: Well-Posed} and apply \eqref{eq: Poincare Inequality} directly. However, it is impossible to encode the ball $\nu$-average as a functional on $\mathcal{H}$ without \textit{a priori} knowledge of the sampling measure $\nu$, which is typically unknown. \\ 

Nevertheless, our approach to proving Theorem \ref{Capacitary Poincare} heavily engages \eqref{eq: Poincare Inequality}. Indeed, we decompose the left hand side of \eqref{eq: Capacitary Poincare} into two terms --- one comparing $g(x)$ to $g_{B(x, r)}$ and another comparing $g_{B(x, r)}$ to $P_x^{\gamma}g$, where here $r$ is chosen so that $B(x, r)$ and $\mathcal{O}_x$ are of comparable size (the latter of which is quantified using upper estimates on the Green function guaranteed by \eqref{eq: VD}, \eqref{eq: Poincare Inequality}, and \eqref{eq: Exit Upper}/\eqref{eq: Exit Lower}). While the first deviation (between $g(x)$ to $g_{B(x, r)}$) can be readily estimated using \eqref{eq: Poincare Inequality}, estimation of the second deviation (between $g_{B(x, r)}$ to $P_x^{\gamma}g$) is more delicate and involves a careful localization argument and the application of \eqref{eq: Cutoff Sobolev}. When compared to \eqref{eq: Poincare Inequality} and in light of \eqref{eq: Growth Condition}, \eqref{eq: Capacitary Poincare} suggests that the ``radius'' of $\mathcal{O}_{x}$ is comparable to $\gamma^{\frac{1}{\beta - \alpha}}$ --- in fact they are \textit{equivalent} (up to a constant factor) when $\mathcal{H} = \mathbb{H}_0^1(\mathbb{R}^d)$ where $\alpha = d$ and $\beta = 2$ (see Example \ref{ex: Euclidean Sobolev example}). \\

We also derive an estimate on the ``second moment'' of the capacitary means $P^{\gamma}_{x}g$. The following result demonstrates that the latter is indeed comparable to $||g||^2_{L^2(\nu)}$ with a slight error depending on $\mathcal{E}(g, g)$ that vanishes as $\gamma \to \infty$. The proof of Lemma \ref{Average Projection} below is similar to that Theorem \ref{Capacitary Poincare} with some additional considerations necessary at the boundary of $\mathcal{X}$. 
\begin{lemma}
\label{Average Projection}
Assume the hypotheses of Theorem \ref{Capacitary Poincare}. Then for $g \in \mathcal{H}$ and sufficiently large $\gamma > 0$, we have that:
\begin{equation*}
    \int_{\mathcal{X}} (P^{\gamma}_x g)^2 d\nu(x) \preceq \gamma^{\frac{\beta}{\beta - \alpha}} \mathcal{E}(g, g) + ||g||^2_{L^2(\mathcal{X}, \nu)}
\end{equation*}
\end{lemma}

\subsection{Convergence Analysis: Proving Theorem \ref{Main Theorem}}
We now discuss the proof of Theorem \ref{Main Theorem}. While the majority of the proof is relegated to appendix \ref{Main Proof} --- here we briefly overview the proof strategy and the main components of the MSE $||\hat{f}_{\lambda, D} - f^{*}||^2_{L^2(\nu)}$. We first define the actions of some operators that will be useful for decomposing our error:
\begin{align}
    C_{\nu}h(x) & \equiv \int G(x, y)h(y)d\nu(y) \label{eq: Covariance Kernel} \\
    \hat{C}_{\nu}h(x) & \equiv \int \gamma_n e_{y, n}(x)\mathcal{E}(\gamma_n e_{y, n}, h) d\nu(y) \label{eq: Cutoff Covariance Kernel} \\
    \hat{C}_{D}h(x) & \equiv \frac{\gamma_n}{n}\sum_{i = 1}^n e_{X_i, n}(x)\mathcal{E}(\gamma_n e_{X_i, n}, h) \label{eq: Sample Cutoff Covariance Kernel}
\end{align}
for all $h \in \mathcal{H}$. Observe that $\hat{C}_{D}$ is simply the empirical version of the population \textit{cutoff covariance} $\hat{C}_{\nu}$ (the subscript denotes dependence on the dataset $\mathcal{D}$). Here, the integral in \eqref{eq: Cutoff Covariance Kernel} is interpreted in the Bochner sense.  \\

We also define the following approximations of the regression function $f^{*} \in \mathcal{H} \cap L^2(\nu)$:
\begin{align}
    \hat{f}_{D, \lambda} \equiv (\hat{C}_D + \lambda)^{-1} \frac{\gamma_n}{n}\sum_{i = 1}^n Y_i e_{i, n} \label{eq: Sample Estimator} \\
    \hat{f}_{\lambda} \equiv (\hat{C}_{\nu} + \lambda)^{-1} \hat{C}_{\nu}f^{*} \label{eq: Regularized Cutoff} \\
    f_{\lambda} \equiv (C_{\nu} + \lambda )^{-1} C_{\nu}f^{*} \label{eq: Regularized True}
\end{align}
Throughout the sequel we will sometimes consider both $L^2(\nu)$-norms and $\mathcal{H}$ norms of the functions above. While by definition, these functions lie in $\mathcal{H}$, we reuse the same notation to denote their equivalence class in $L^2(\nu)$, when the setting is clear from context. As usual, we always assume we are working with a fixed quasi-continuous representative of a given function. \\
\subsubsection{Bias, Variance, and Approximation}
We can now decompose the mean-square error as:
\begin{equation}
\label{eq: Three Part Breakdown}
    ||f^{*} - \hat{f}_{D, \lambda}||^2_{L^2(\nu)} \leq 4\Big(\underbrace{||f^{*} - f_{\lambda}||^2_{L^2(\nu)}}_{\text{bias}} + \underbrace{||\hat{f}_{D, \lambda} - \hat{f}_{\lambda}||^2_{L^2(\nu)}}_{\text{variance}} + \underbrace{||f_{\lambda} - \hat{f}_{\lambda}||^2_{L^2(\nu)}}_{\text{approximation}}\Big)
\end{equation}
We expound on each of these components below:
\begin{itemize}
\item \textbf{Bias}: The first term in \eqref{eq: Three Part Breakdown} is the ordinary regularization bias we incur due to the ridge penalty
\item \textbf{Variance}: The second term in \eqref{eq: Three Part Breakdown} is the variance, which characterizes the concentration of the empirical cutoff covariance around its population counterpart. Due to our cutoff renormalization of the Green's kernel (see section \ref{Green Truncation}), both covariance operators are bounded over the Hilbert space $\mathcal{H}$ and their concentration in $||\cdot||_{\mathcal{H}}$ can be studied using standard techniques. However, relating $||\cdot||_{L^2(\nu)}$ and $||\cdot||_{\mathcal{H}}$ involves the \textit{true} covariance operator $C_{\nu}$ (see \eqref{eq: Covariance Reproducing}) and hence we will also have to employ the capacitary Poincar\'e inequality (Theorem \ref{Capacitary Poincare}) to compare $C_{\nu}$ to $\hat{C}_{\nu}$
\item \textbf{Approximation} While the bias compares $f^{*}$ and $f_{\lambda}$ and the variance $\hat{f}_{D, \lambda}$ and $\hat{f}_{\lambda}$, there is still a gap between $f_{\lambda}$ and $\hat{f}_{\lambda}$ that stems from the cutoff approximation to the Green's kernel (renormalization). This error characterizes the crux of our approach and we employ the Poincar\'e inequalities of section \ref{Capacitary Poincare} in full force. 
\end{itemize}

The ridge bias in the first term will be controlled using a standard spectral argument and exhibits the classical linear decay in $\lambda$. The estimation of the variance and approximation error are significantly more involved. As we might expect, the variance increases in the presence of insufficient regularization in either the Green function (choosing $\gamma$ too large) or ridge penalty (choosing $\lambda$ to small). Perhaps, the most interesting term is the final approximation error, which responds differently to each type of regularization. Indeed, this term behaves like a bias with respect to Green function renormalization, encouraging a higher $\gamma$ (and hence smaller obstacles), while advocating for the shrinkage provided by the ridge (i.e. higher $\lambda$). The latter phenomenon reflects the cooperation between the two regularizers, with shrinked elements of $\mathcal{H}$ being easier to approximate. 


\section{Random Obstacle ERM}
\label{Random Obstacle ERM}
In the final section, we consider the role of random obstacle renormalization in pure empirical risk minimization, where we are interested in solving the following constrained least-squares problem:

\begin{equation}
\label{Pure ERM}
    \hat{f}_{D} = \arg \min_{g \in \mathcal{H}_{M, \gamma_n}} \frac{1}{n}\sum_{i = 1}^n (Y_i - g(X_i))^2
\end{equation}
where:
\begin{equation}
\label{eq: Reg Function Class}
    \mathcal{H}_{M, \gamma_n} = \{\tilde{h}(x) \equiv \mathcal{E}(h, \gamma_n e_x^{\gamma_n}): h \in M\mathbb{B}(\mathcal{H})\}
\end{equation}
is the ``smoothed'' Dirichlet $M$-ball obtained by renormalizing the elements via the equilibrium measure ($\gamma_n$ is a data-driven threshold as before and $e_{x}^{\gamma_n}$ is the equilibrium potential of the obstacle $\mathcal{O}_{x, n}(\gamma_n)$ in \eqref{eq: Obstacle Def}). We note this regularization transforms $M\mathbb{B}(\mathcal{H})$ into a bounded subset of $L^{\infty}$. In order to fully avail ourselves to empirical process tools, we additionally need the following curvature assumption on the metric measure Dirichlet space. For $h \in \mathcal{H}$, let $|\nabla h| = \sqrt{\frac{d\Gamma(h, h)}{d\mu}}$. We suppose:
\begin{equation}
\tag{BKE}
\label{eq: Bakry-Emery}
    |\nabla P_t h(x)|^2 \leq CP_{ct} |\nabla h|^2(x)
\end{equation}
for all $h \in \mathcal{H}$, $x \in \mathcal{M}$, and $t \geq 0$ and some $C, c > 0$. Following the nomenclature in \cite{coulhon2020gradient}, we call \eqref{eq: Bakry-Emery} the generalized Bakry-Emery condition. Geometrically, \eqref{eq: Bakry-Emery} implies that the MMD space $(\mathcal{M}, \mu, d, \mathcal{E})$ has nonnegative curvature in the Bakry-Emery sense, and reduces to the standard Bakry-Emery condition when $c = 1$. Note \eqref{eq: Bakry-Emery} implies that the energy measure $\Gamma(h, h)$ is absolutely continuous with respect to the ambient measure $\mu$, and hence rules out the fractal setting \cite{kajino2020singularity}. The assumption \eqref{eq: Bakry-Emery} is needed to ensure the local Lipschitzness of harmonic functions \cite{coulhon2020gradient}, which will be critical in studying the modulus of continuity of the Gaussian process indexed by \eqref{eq: Reg Function Class}. We also introduce the intrinsic metric $d_{\mathcal{E}}$ on $(\mathcal{M}, \mu, \mathcal{E})$:
\begin{equation}
\label{eq: Intrinsic Metric}
    d_{\mathcal{E}}(x, y) = \sup \{f(x) - f(y): f \in \mathcal{H}, \nabla f \leq 1 \hspace{2mm} a.e. \}
\end{equation}
for $x, y \in \mathcal{M}$. We also assume $\beta = 2$ in \eqref{eq: Growth Condition}, i.e. there exists a $r_0 \in (0, 1)$ such that for all $r \in (0, r_0)$:
\begin{equation}
\label{eq: Walk Dimension 2}
    \Psi(r) \asymp r^2
\end{equation}
Note that \eqref{eq: Walk Dimension 2} comes with minimal loss of generality, as \eqref{eq: Bakry-Emery} requires that $\frac{\Psi(r\lambda^{-1})}{\Psi(r)} \succeq \lambda^{-2}$ as $r \to 0$ and $\lambda \to \infty$ by Theorem 2.13(a) in \cite{kajino2020singularity}, while the combination of \eqref{eq: Poincare Inequality}, \eqref{eq: VD}, and \eqref{eq: Cutoff Sobolev} implies that $\frac{\Psi(r)}{\Psi(r\lambda^{-1})} \succeq \lambda^{2}$ by Theorem 1.6 in \cite{murugan2020length}. Furthermore, by Theorem 2.13(b) in \cite{kajino2020singularity}, we have the bi-Lipschitz equivalence of $d_{\mathcal{E}}$ and $d$, i.e. there exists a $r_2 > 0$ and constant $C_5 > 1$ such that:
\begin{equation}
\label{Metrics are Equivalent}
    C_5^{-1}d(x, y) \leq d_{\mathcal{E}}(x, y) \leq C_5d(x, y)
\end{equation}
for $d(x, y) \wedge d_{\mathcal{E}}(x, y) \leq r_2$. When $(\mathcal{M}, d)$ is a length space (see e.g. \cite{murugan2020length} for a definition) \eqref{Metrics are Equivalent} holds with $r_2 = \infty$. In this section, we will always assume $(\mathcal{M}, d)$ is a length space for convenience. \\

For simplicity, in the sequel, we will suppose $M = 1$ in \eqref{eq: Reg Function Class} and denote $\mathcal{H}_{1, \gamma}$ as simply $\mathcal{H}_{\gamma}$. The crucial auxiliary result in this section is the following upper estimate for the metric entropy of $\mathcal{H}_{\gamma}$:

\begin{proposition}
\label{Covering Numbers}
\begin{equation*}
    \log \mathcal{N}(\epsilon, \mathcal{H}_{\gamma}, \|\cdot\|_{\infty}) \preceq \gamma^{\frac{\alpha(\alpha - 1)}{(\alpha + 1)(\alpha - 2)}}\epsilon^{-\frac{2\alpha}{\alpha + 1}}
\end{equation*}
where $\mathcal{N}(\epsilon, \mathcal{H}_{\gamma}, \|\cdot\|_{\infty})$ is the $\epsilon$-covering number of $\mathcal{H}_{\gamma}$ in the $L^{\infty}(\mathcal{X})$ norm.
\end{proposition}

The proof of Proposition \ref{Covering Numbers} is provided in Appendix \ref{ERM Proofs}. Our approach involves studying the Gaussian measure $\lambda_{\mathcal{H}}$ with Cameron-Martin space $\mathcal{H}$, and exploiting the interaction \cite{kuelbs1993metric} between metric entropy and the small ball probabilities $\lambda_{\mathcal{H}}(K_{\gamma}(\epsilon))$, where $K_{\gamma}(\epsilon)$ is the convex body:
\begin{equation}
\label{eq: Regularized Sup Ball}
K_{\gamma}(\epsilon) = \{h \in \mathcal{H}: \sup_{x \in \mathcal{X}} \mathcal{E}(h, \gamma e^{\gamma}_{x}) \leq \epsilon\}
\end{equation}
We estimate $\lambda_{\mathcal{H}}(K_{\gamma}(\epsilon))$ using a classical technique of Talagrand \cite{talagrand1993new} involving chaining and Sidak's inequality for Gaussian measures. We suspect that Proposition \ref{Covering Numbers} and its proof may be of independent interest, and potentially useful in the studying thick point sets of Gaussian free fields over general metric spaces (in the spirit of \cite{hu2010thick} for the 2D Euclidean free field). 

Equipped with control of the metric entropy, we readily derive the following risk bound:

\begin{theorem}
\label{ERM Risk}
Suppose \eqref{eq: VD}, \eqref{eq: RVD}, \eqref{eq: VD/RVD}, \eqref{eq: Poincare Inequality}, \eqref{eq: Exit Upper}, \eqref{eq: Exit Lower}, \eqref{eq: Bakry-Emery}, and \eqref{eq: Walk Dimension 2} hold. Further assume $\|f^{*}\|_{L^{\infty}(\mathcal{X})} < \infty$ and that noise is subgaussian $\epsilon \sim \text{SG}(\rho^2)$ in \eqref{eq: Additive Noise Model}. Then, if $f^{*} \in \mathbb{B}(\mathcal{H})$ and $\gamma_n \asymp n^{\frac{\alpha - 2}{2\alpha}}$:
\begin{equation}
\label{eq: Pure ERM Consistency}
    ||\hat{f}_{D} - f^{*}||^2_{L^2(\nu)} \preceq \log \delta^{-1} \cdot n^{-\frac{1}{\alpha}}
\end{equation}
with probability $1 - 2\delta$.
\end{theorem}

We observe that relative to Theorem \ref{Main Theorem}, Theorem \ref{ERM Risk} features a sharper dependence on the confidence level in \eqref{eq: Pure ERM Consistency}, with $\log \delta^{-1}$ replacing $\delta^{-1}$ in \eqref{eq: Main Theorem Statement}. This sharpening is enabled by the additional domain smoothness provided by \eqref{eq: Bakry-Emery}, which implies the H\"older continuity of the elements of the smoothed Dirichlet ball in \eqref{eq: Reg Function Class} (see proof of Proposition \ref{Covering Numbers} in Appendix \ref{ERM Proofs}). This additional control on the regularity of $\mathcal{H}_{\gamma_n}$ and the subgaussianity of the noise permit a uniform "worst-case" concentration analysis over $\mathcal{H}_{ \gamma_n}$, leading to an improved dependence on the confidence level $\delta > 0$ in \eqref{eq: Pure ERM Consistency}. \\

However, more critically, the upper bound in \eqref{eq: Pure ERM Consistency} exhibits a gap with respect to the optimal rate of $n^{-\frac{2}{\alpha + 2}}$ obtained in \eqref{eq: Main Theorem Statement} of Theorem \ref{Main Theorem}. We believe that this gap is expected, and reflects the suboptimality of ERM with respect to the statistical complexity of the Dirichlet space. \\

Indeed, the gap is reminiscent of the well-observed rate sub-optimality of empirical risk minimization for non-Donsker classes, first noted in \cite{birge1993rates} and more recently addressed in \cite{kuchibhotla2022least, kur2020suboptimality, han2021set}. In fact, for general $\mathcal{H}$, we see from the heat kernel bounds implied by \eqref{eq: VD}, \eqref{eq: Poincare Inequality}, \eqref{eq: Exit Upper}, \eqref{eq: Exit Lower} and Theorem 1.2 in \cite{grigor2015generalized}, Theorem 1 in \cite{varopoulos1985hardy}, the definition of $\nu$ in \eqref{eq: Sampling Measure}, and Theorem 1.5 in \cite{ben2007sobolev} that the eigenvalues of the covariance operator $C_{\nu}$ obey a type of Weyl's law:
\begin{equation*}
    \lambda_k(C_{\nu}) \preceq k^{-\frac{\beta}{\alpha}}
\end{equation*}
Then, from the relationship between entropy numbers and eigenvalues \cite{edmunds1996function}, we can easily deduce that the $L^2(\nu)$-metric entropy of the unit ball in $\mathcal{H}(\mathcal{X})$ \footnote{$\mathcal{H}(\mathcal{X})$ denotes the trace of $\mathcal{H}$ to $(\mathcal{X}, \nu)$, see section 6.2 of \cite{fukushima2010dirichlet}} scales like:
\begin{equation}
\label{eq: MS Metric Entropy}
    \log \mathcal{N}(\epsilon, \mathbb{B}(\mathcal{H}(\mathcal{X})), \|\cdot\|_{L^2(\nu)}) \preceq \epsilon^{-\frac{2\alpha}{\beta}}
\end{equation}
Clearly $\mathbb{B}(\mathcal{H}(\mathcal{X}))$ is not pregaussian, and hence not Donsker, for $\alpha \geq \beta$ (precisely the transient regime we consider). Substituting $\beta = 2$ in \eqref{eq: MS Metric Entropy}, we notice that the upper bound in Theorem \ref{ERM Risk} matches the expected upper bound \cite{birge1993rates} for \textit{supercritical} non-Donsker classes, when the bracketing and covering $L^2(\nu)$ metric entropies are asymptotically equivalent. We emphasize that this is merely a stylistic comparison, the classical techniques from \cite{birge1993rates} and related works \cite{gine2016mathematical, van1996weak} cannot be applied to our setting as the uniform and bracketing metric entropy of  $\mathbb{B}(\mathcal{H})$ are not well-defined due to its elements lacking pointwise definition. Hence, Theorem \ref{ERM Risk} suggests that random obstacle renormalization extends the analysis of non-Donsker classes beyond the supercritical regime.

\section{Discussion} 
In this paper, we consider nonparametric estimation over general metric measure Dirichlet spaces. Unlike the more commonly studied reproducing kernel Hilbert space, whose elements may be defined pointwise, a Dirichlet space typically only contains equivalence classes, i.e. its elements are only unique almost everywhere. This lack of pointwise definition presents significant challenges in the context of nonparametric estimation, leading the classical ridge regression problem to be ill-posed. In this paper, we develop a new technique for renormalizing the square loss by replacing pointwise evaluations with certain \textit{local means} around the boundaries of obstacles centered at each data point. The resulting renormalized empirical risk functional is well-posed and even admits a representer theorem in terms of certain equilibrium potentials, which are just truncated versions of the associated Green function, cut-off at a data-driven threshold. We demonstrate the global, out-of-sample rate optimality of the sample minimizer, and derive an adaptive upper bound on its convergence rate that highlights the interplay of the analytic, geometric, and probabilistic properties of the Dirichlet form. Our framework notably does not require the smoothness of the underlying space, and is applicable to both manifold and fractal settings. To the best of our knowledge, this is the first paper to obtain optimal out-of-sample convergence guarantees in the framework of general metric measure Dirichlet spaces. \\

There are several further avenues for investigation. Perhaps the most pressing direction involves exploring the additional structure required on $\nu$ to close the gap between the convergence rates obtained for renormalized ridge regression in Theorem \ref{Main Theorem} (optimal) and constrained ERM (over a renormalized Dirichlet ball) in Theorem \ref{ERM Risk} (suboptimal). In particular, a stronger isoperimetric condition on $\nu$, such as an $L^1(\nu)$-Poincar\'e inequality, may enable a sharper entropic analysis such as in \cite{han2021set}. Another promising direction involves exploring whether the optimal rates obtained by renormalized ridge regression may be extended to smoother, subcritical Sobolev spaces (i.e. $\mathcal{H}^s(\mathcal{M})$ for $1 < s < \frac{\alpha}{\beta}$) by interpolating between the Dirichlet space ($\mathcal{H}^1(\mathcal{M})$) and a sufficiently regular, supercritical Sobolev space that enjoys the standard $n^{-\frac{1}{2}}$ consistency. 
\appendix

\section{Proof of Theorem \ref{Capacitary Poincare} and Lemma \ref{Average Projection}}
\label{Poincare Proofs}
\begin{proof}[Proof of Theorem \ref{Capacitary Poincare}]
Firstly, we note that it is sufficient to prove the theorem for $g \in C_{0}(\mathcal{M}) \cap \mathcal{H}$. Indeed, suppose \eqref{eq: Capacitary Poincare} holds for $g \in C_{0}(\mathcal{M}) \cap \mathcal{H}$. Then, by the regularity of $\mathcal{E}$, we have that for any $g \in \mathcal{H}$, there is an approximating sequence of $\{g_i\}_{i = 1}^{\infty} \in C_{0}(\mathcal{M}) \cap \mathcal{H}$ such that $\mathcal{E}(g_i - g, g_i - g) \to 0$ as $i \to \infty$. Hence, we have that:
\begin{align}
    \int_{\mathcal{X}} (g(x) - P^{\gamma}_{x}g)^2d\nu(x) & \leq 4\int_{\mathcal{X}} (g_i(x) - P^{\gamma}_{x}g_i)^2d\nu(x) + 4\int_{\mathcal{X}} (g_i - g)^2d\nu + 4\int_{\mathcal{X}} (P^{\gamma}_x g - P^{\gamma}_x g_i)^2 d\nu(x) \nonumber \\
    & \leq 4\int_{\mathcal{X}} (g_i(x) - P^{\gamma}_{x}g_i)^2d\nu(x) + 4||g - g_i||^2_{L^2(\nu)} + 4\gamma\mathcal{E}(g_i - g, g_i - g) \label{eq: Norm of Obstacle} \\
    & \preceq \gamma^{\frac{\beta}{\beta - \alpha}} \mathcal{E}(g_i, g_i) + 4||g - g_i||^2_{L^2(\nu)} + 4\gamma\mathcal{E}(g_i - g, g_i - g) \nonumber 
\end{align}
where in \eqref{eq: Norm of Obstacle} we have applied the fact that $|P^{\gamma}_{x}h| = |\mathcal{E}(h, \gamma e_x)| \leq \sqrt{\gamma \mathcal{E}(h, h)}$ for any $h \in \mathcal{H}$ by construction. Now, choosing balls $B_1, B_2 \subset \mathcal{M}$ such that $\mathcal{X} \subset B_1 \subset B_2$, and applying the Faber-Krahn inequality in $B_2$ (follows from \eqref{eq: VD}, \eqref{eq: Poincare Inequality}, \eqref{eq: Exit Upper}/\eqref{eq: Exit Lower} by Theorem 1.2 in \cite{grigor2015generalized} and Theorem 3.11 in \cite{grigor2012two}), to $h\phi$ for $\phi \in \text{cutoff}(B_1, B_2)$ given by \eqref{eq: Cutoff Sobolev}, we can readily show:
\begin{equation*}
    ||h||^2_{L^2(\nu)} \leq C\mathcal{E}(h, h)
\end{equation*}
for some constant $C > 0$ independent of $h \in \mathcal{H} \cap L^2(\nu)$ (and depending only on $\text{diam}(\mathcal{X})$). Hence, we have $||g - g_i||^2_{L^2(\nu)} \preceq \mathcal{E}(g-g_i, g-g_i)$ in \eqref{eq: Norm of Obstacle}. Passing the limit $i \to \infty$, we obtain \eqref{eq: Capacitary Poincare} for $g \in \mathcal{H}$. \\

We will first demonstrate that there exists a $r > 0$ such that $\mathcal{O}_{x} \in B(x, r)$ for all $x \in \mathcal{X}$. Since $\gamma = \text{cap}(\mathcal{O}_x)^{-1}$ is sufficiently large, we can without loss of generality suppose $d(x, y) \leq 1$ for $y \in \mathcal{O}_x$. Let $R = \text{diam}(\mathcal{X})$. Then, we have by Theorem 7.5 in \cite{grigor2014heat} that:
\begin{align}
    G(x, y) & \leq C\int_{\frac{d(x, y)}{4}}^{\infty} \frac{\Psi(s)}{sV(x, s)} ds \nonumber \\
    & \leq \int_{\frac{d(x, y)}{4}}^{1} \frac{\Psi(s)}{sV(x, s)} ds + \int_{1}^{\infty} \frac{\Psi(s)}{sV(x, s)} ds \nonumber \\
    & \leq \frac{C\Psi(1)}{V(x, 1)}\int_{\frac{d(x, y)}{4}}^{1} \frac{s^{\beta}}{s^{\alpha + 1}} ds + \frac{C\Psi(1)}{V(x, 1)}\int_{1}^{\infty} \frac{s^{\beta'}}{s^{\alpha' + 1}} ds  \label{eq: Apply VD + Growth} \\
    & \leq \frac{CR^{\alpha}\Psi(1)d(x, y)^{\beta - \alpha}}{\mu(\mathcal{X})} \label{eq: Apply Diameter VD}
\end{align}
where in \eqref{eq: Apply VD + Growth} we have applied the growth conditions of $\Psi$ \eqref{eq: Growth Condition} and \eqref{eq: VD}/\eqref{eq: RVD}, and in \eqref{eq: Apply Diameter VD} we have again applied volume doubling and the definition of $R$. Hence, we obtain, for $y \in \mathcal{O}_x$:
\begin{align}
d(x, y) & \leq K_1G(x, y)^{\frac{1}{\beta - \alpha}} \nonumber \\
& \leq K_1\gamma^{\frac{1}{\beta - \alpha}} \label{eq: Radial Upper Bound}
\end{align}
for some $K_1 > 0$ independent of $\gamma, x$, or $y$. For the remainder of the proof let $r = K_1\gamma^{\frac{1}{\beta - \alpha}}$. We first choose a $r$-net $\{x_i\}_{i = 1}^N$ of $\mathcal{X}_r$ such that $B(x_i, 3\sigma^{-1}r)$ satisfy the bounded overlap property (this always follows by volume doubling). Then, we observe that:
\begin{equation}
\label{eq: Variance Breakdown}
    \int_{B(x_i, r)} (g - P_x g)^2 d\mu \leq 2  \int_{B(x_i, r)} (g(x) - g_{B(x_i, r)})^2 d\mu(x) + 2\int_{B(x_i, r)} (P_xg - g_{B(x_i, r)})^2 d\mu(x)
\end{equation}
We can estimate the first term above using \eqref{eq: Poincare Inequality}. To estimate the second term, we use \eqref{eq: Cutoff Sobolev}. Namely, fix $x \in B(x_i, r)$ and let $\phi \in \text{Cutoff}(B(x, r), 2B(x, r))$ be the cutoff function in \eqref{eq: Cutoff Sobolev} for $h = g - g_{B(x_i, r)}$. Then, since $\mathcal{O}_x \in B(x, r)$ we have that:
\begin{align}
    \frac{1}{\nu_x(\partial \mathcal{O}_x)}\int (g(y) - g_{B(x_i, r)})d\nu_x(y) & = \frac{1}{\nu_x(\partial \mathcal{O}_x)}\int (g(y) - g_{B(x_i, r)}) \phi(y) d\nu_x(y) \nonumber \\
    & = \frac{1}{\nu_x(\partial \mathcal{O}_x)} \mathcal{E}((g - g_{B(x_i, r)}) \phi, e_x) \nonumber \\
    & = \frac{1}{\nu_x(\partial \mathcal{O}_x)} \int_{B(x, 2r)} d\Gamma((g - g_{B(x_i, r)}) \phi, e_x) \label{eq: Support Condition} \\
    & \leq \sqrt{\gamma} \Big(\int_{B(x, 2r)} d\Gamma((g - g_{B(x_i, r)}) \phi, (g - g_{B(x_i, r)}) \phi)\Big)^{\frac{1}{2}} \label{eq: Cauchy Schwarz} \\
    & \leq \sqrt{\gamma} \sqrt{\frac{3}{2}\int_{B(x, 2r)} (g - g_{B(x_i, r)})^2 d\Gamma(\phi, \phi) + 4\int_{B(x, 2r)} \phi^2 d\Gamma(g, g)} \label{eq: Product Rule} \\
    & \leq \sqrt{\gamma} \sqrt{\Big(\frac{3C_4}{2\Psi(r)}\Big)\int_{B(x, 2r)} (g - g_{B(x_i, r)})^2 d\mu + (4 + 1.5C_3)\int_{B(x, 2r)} d\Gamma(g, g)} \label{eq: CSA} 
\end{align}
where \eqref{eq: Support Condition} follows from $\text{supp}(e_x), \text{supp}(\phi) \subset B(x, 2r)$, in \eqref{eq: Cauchy Schwarz} we have applied Cauchy-Schwarz and the fact $\mathcal{E}(e_x, e_x) = \nu_x(\partial \mathcal{O}_x) = \gamma^{-1}$, in \eqref{eq: Product Rule} we have applied the Leibniz rule for the carr\'e du champ, and in \eqref{eq: CSA} we have applied \eqref{eq: Cutoff Sobolev}. Now, observe  that:
\begin{align}
    (g_{B(x, 2r)} - g_{B(x_i, r)})^2 & \leq \frac{1}{V(x, 2r)V(x_i, r)} \int_{B(x_i, r)}\int_{B(x, 2r)} (g(y) - g(z))^2d\mu d\mu \nonumber \\
    & \preceq \frac{1}{V(x_i, 3r)V(x_i, 3r)} \int_{B(x_i, r)}\int_{B(x_i, 3r)} (g(y) - g(z))^2d\mu d\mu \label{eq: Apply Volume Doubling} \\
    & \preceq \frac{\Psi(r)}{V(x_i, 3r)} \int_{B(x_i, 3\sigma^{-1}r)} d\Gamma(g, g) \label{eq: Revised Poincare}
\end{align}
where in \eqref{eq: Apply Volume Doubling} we have applied volume doubling, and \eqref{eq: Revised Poincare} follows from an alternative version of the Poincar\'e inequality (see e.g. Lemma 4.1 in \cite{grigor2015generalized}) and \eqref{eq: Growth Condition}. Hence, we obtain that:
\begin{align}
    \int_{B(x, 2r)} (g - g_{B(x_i, r)})^2 d\mu & \leq 2\int_{B(x, 2r)} (g - g_{B(x, 2r)})^2 d\mu + 2V(x, 2r)(g_{B(x, 2r)} - g_{B(x_i, r)})^2 \nonumber \\
    & \preceq 2C\Psi(r)\int_{B(x, 2\sigma^{-1}r)} d\Gamma(g, g) + \Psi(r) \int_{B(x_i, 3\sigma^{-1}r)} d\Gamma(g, g) \label{eq: Apply Difference of Means} \\
    & \preceq \Psi(r) \int_{B(x_i, 3\sigma^{-1}r)} d\Gamma(g, g) \label{eq: Final Energy Bound}
\end{align}
where in \eqref{eq: Apply Difference of Means} we have applied \eqref{eq: Poincare Inequality}, \eqref{eq: Revised Poincare} and  $B(x, 2r) \subset B(x_i, 3r)$. Substituting \eqref{eq: Final Energy Bound} back into \eqref{eq: CSA}, we obtain:
\begin{align}
     \frac{1}{\nu_x(\partial \mathcal{O}_x)}\int (g - g_{B(x_i, r)})d\nu_x & \leq \sqrt{\gamma} \sqrt{\Big(\frac{3C_4}{2\Psi(r)}\Big)\int_{B(x, 2r)} (g - g_{B(x_i, r)})^2 d\mu + (4 + 1.5C_3)\int_{B(x, 2r)} d\Gamma(g, g)} \nonumber \\
    & \preceq \sqrt{\gamma} \sqrt{\int_{B(x_i, 3\sigma^{-1}r)} d\Gamma(g, g) + (4 + 1.5C_3)\int_{B(x, 2r)} d\Gamma(g, g)} \\
    & \preceq \sqrt{\gamma} \sqrt{\int_{B(x_i, 3\sigma^{-1}r)} d\Gamma(g, g)} \label{eq: Energy Upper Bound}
\end{align}
Then, applying \eqref{eq: Poincare Inequality} again to the first term of \eqref{eq: Variance Breakdown} and substituting \eqref{eq: Energy Upper Bound} for the second term, we have that:
\begin{equation*}
    \int_{B(x_i, r)} (g - P_x g)^2 d\mu \leq 2\Psi(r)\int_{B(x_i, \sigma^{-1}r)}d\Gamma(g, g) + \gamma V(x_i, r)\int_{B(x_i, 3\sigma^{-1}r)} d\Gamma(g, g)
\end{equation*}
Now, observe that by \eqref{eq: VD/RVD} and \eqref{eq: Radial Upper Bound}, we have that:
\begin{align*}
    \gamma V(x_i, r) & \preceq \gamma r^{\alpha}  \\
    & \preceq \gamma^{1 + \frac{\alpha}{\beta - \alpha}} \\
    & = \gamma^{\frac{\beta}{\beta - \alpha}}
\end{align*}
by the definition of $r$. Likewise:
\begin{align*}
    \Psi(r) & \preceq \Psi(\gamma^{\frac{1}{\beta - \alpha}})  \\
    & \preceq \gamma^{\frac{\beta}{\beta - \alpha}}
\end{align*}
by the growth assumptions \eqref{eq: Growth Condition}. Hence, putting this all together, we have:
\begin{equation*}
    \int_{B(x_i, r)} (g - P_x g)^2 d\mu \preceq \gamma^{\frac{\beta}{\beta - \alpha}} \int_{B(x_i, 3\sigma^{-1}r)} d\Gamma(g, g)
\end{equation*}
Summing over the $i \in [n]$ and noting that, by the bounded overlap property, only a fixed number (independent of $r$) of $B(x_i, 3\sigma^{-1}r)$ may intersect at a given point in $\mathcal{X}$, we have:
\begin{equation}
\label{eq: Average over Interior}
    \int_{\mathcal{X}} (g - P_x g)^2 d\mu \preceq \gamma^{\frac{\beta}{\beta - \alpha}}\mathcal{E}(g, g)
\end{equation}
Now, dividing both sides by $\mu(\mathcal{X})$ we obtain our result. 
\end{proof}

\begin{proof}[Proof of Lemma \ref{Average Projection}]
The proof strongly resembles that of Theorem \ref{Capacitary Poincare} above so we provide a sketch and omit the details. Again, we may assume $g \in C_0(\mathcal{M}) \cap \mathcal{H}$.  Let $r = K_1\gamma^{\frac{1}{\beta - \alpha}}$ be the same as in \eqref{eq: Radial Upper Bound} in the proof of Theorem \ref{Capacitary Poincare}, and begin with a $r$-net $\{x_i\}_{i = 1}^N$ of $\mathcal{X} = \text{supp}(\nu)$, such that $\{B(x_i, 3\sigma^{-1}r)\}_{i = 1}^N$ enjoy the bounded overlap property. Let $x \in B(x_i, r)$ and let $\phi \in \text{Cutoff}(B(x, r), 2B(x, r))$ be the cutoff function in \eqref{eq: Cutoff Sobolev} for $h = g - g_{B(x_i, r)}$. Since $\mathcal{O}_x \subset B(x, r)$, we can write:
\begin{align}
    P^{\gamma}_x g & = \frac{1}{\nu_x(\partial \mathcal{O}_x)}\int g(y) d\nu_x(y) \nonumber \\
    & = \frac{1}{\nu_x(\partial \mathcal{O}_x)}\int g(y) \phi(y) d\nu_x(y) \nonumber \\
    & = \frac{1}{\nu_x(\partial \mathcal{O}_x)} \mathcal{E}(g \phi, e_x) \nonumber \\
    & = \frac{1}{\nu_x(\partial \mathcal{O}_x)} \int_{B(x, 2r)} d\Gamma(g \phi, e_x) \nonumber \\
    & \leq \sqrt{\gamma} \Big(\int_{B(x, 2r)} d\Gamma(g \phi, g \phi)\Big)^{\frac{1}{2}} \label{eq: Cauchy Schwarz 2} \\
    & \leq \sqrt{\gamma} \sqrt{\frac{3}{2}\int_{B(x, 2r)} g^2 d\Gamma(\phi, \phi) + 4\int_{B(x, 2r)} \phi^2 d\Gamma(g, g)} \label{eq: Product Rule 2} \\
    & \leq \sqrt{\gamma} \sqrt{\Big(\frac{3C_2}{2\Psi(r)}\Big)\int_{B(x, 2r)} g^2 d\mu + (4 + 1.5C_1)\int_{B(x, 2r)} d\Gamma(g, g)} \label{eq: CSA 2} 
\end{align}
where again \eqref{eq: Cauchy Schwarz 2} follows from the Cauchy-Schwarz inequality, \eqref{eq: Product Rule 2} follows from the Leibniz rule for the energy measure, and \eqref{eq: CSA 2} follows from the cutoff Sobolev inequality \eqref{eq: Cutoff Sobolev}. Now squaring both sides, noting again that $\gamma V(x_i, r) \preceq \gamma^{\frac{\beta}{\beta - \alpha}}$ and $\Psi(r) = \Psi(K_1 \gamma^{\frac{1}{\beta - \alpha}}) \succeq \gamma^{\frac{\beta}{\beta - \alpha}}$ by \eqref{eq: VD/RVD} and \eqref{eq: Growth Condition}, we obtain:
\begin{equation*}
    \int_{B(x_i, r)} (P^{\gamma}_x g)^2 d\mu(x) \preceq \int_{B(x_i, 3r)} g^2 d\mu +  \Psi(r) \int_{B(x_i, 3r)} d\Gamma(g, g)
\end{equation*}
Summing over $i \in [N]$ and recalling that the family $\{B(x_i, 3r)\}_{i = 1}^N$ have the bounded overlap property (by volume doubling), we obtain: 
\begin{equation}
\label{eq: Energy + L2 with Overlap}
    \int_{\mathcal{X}} (P^{\gamma}_x g)^2 d\mu(x) \preceq \int_{\mathcal{X}_{3r}} g^2 d\mu +  \Psi(r) \int_{\mathcal{X}_{3r}} d\Gamma(g, g)
\end{equation}
where $\mathcal{X}_s \equiv \{x \in \mathcal{M}: d(x, \mathcal{X}) \leq s\}$ denotes the $s$-neighborhood of $\mathcal{X}$. Recall that $\mathcal{X} = \text{supp}(\nu)$. So in our final step, we wish to estimate $\int_{\mathcal{X}_{3r} \setminus \mathcal{X}} g^2 d\mu$ in terms of $||g||_{L^2(\nu)}$ and $\mathcal{E}(g, g)$. Recall that $\mathcal{X}^{\circ}$ is an $A$-uniform domain in $\mathcal{M}$ by assumption and let $\mathcal{X}^r = \{x \in \mathcal{M}: 0.5 A^{-1}r \leq d(x, \mathcal{X}^c) \leq Ar\}$ (note that since $r \to 0$ as $\gamma \to \infty$, we can w.l.o.g. suppose $r > 0$ is sufficiently small). Consider a $\frac{r}{2}$-net $V$ of this set, and observe that, by construction, the family $\Big\{B\Big(z,  \frac{r}{4}\Big)\Big\}_{z \in V}$ is disjoint. Now, consider a $r$-net $W$ of $X_{3r} \setminus X^{\circ}$. For each $w \in W$, let $z_{w} \in V$ denote the element of the net $V$ that is closest to $w$, i.e. $d(w, z_w) = \min_{z \in V} d(w, z)$. Observe that for a given $z \in V$, there are only finitely many $w \in W$ such that $z = z_{w}$. Indeed, since $w \subset \mathcal{X}_{3r} \setminus \mathcal{X}^{o}$ and $V$ is an $\frac{r}{2}$-net of $\mathcal{X}^r$, it follows that $d(w, z_{w}) \leq \frac{9r}{2}$. This can be seen by letting $\xi \in \partial \mathcal{X}$ be such that $\xi = \text{arg} \min_{x \in \mathcal{X}} d(w, x)$, and observing that since $\mathcal{X}^{\circ}$ is $A$-uniform, it satisfies a corkscrew condition, i.e. there exists a $y \in \mathcal{X}^{\circ}$ such that $d(\xi, y) = r$ and $d(y, \mathcal{X}^{c}) > 0.5A^{-1}r$ (Lemma 2.6 in \cite{kajino2023heat}). Thus, $y \in \mathcal{X}^r$ and hence, there exists a $\tilde{z}_w \in V$, with $d(y, \tilde{z}_w) \leq \frac{r}{2}$. Then, $d(w, z_w) \leq d(w, \tilde{z}_w) \leq 3r + r + 
\frac{r}{2} = \frac{9r}{2}$ by the triangle inequality. Thus, if $w, \tilde{w} \in W$ are such that $z = z_{w} = z_{\tilde{w}}$, then it follows that $w, \tilde{w} \in B\big(z, \frac{9r}{2}\big)$, and hence since $d(w, \tilde{w}) \geq r$, by volume doubling we have that there are only finitely many such distinct pairs, and this number is independent of $r$. Thus, we may write:
\begin{align}
    \int_{\mathcal{X}_{3r} \setminus \mathcal{X}} g^2 d\mu & \leq \sum_{w \in W} \int_{B(w, r)} g^2 d\mu \nonumber \\
    & \preceq \sum_{w \in W} \int_{B(w, r)} (g - g_{B(w, r)})^2 d\mu + \int_{B(w, r)} (g_{B(z_w, 0.25r)} - g_{B(w, r)})^2 d\mu + V(w, r)g^2_{B(z_w, 0.25r)} \nonumber \\
    & \preceq \sum_{w \in W} \Psi(r)\int_{B(w, \sigma^{-1}r)} d\Gamma(g, g)  + \frac{\Psi(r)V(w, r)}{V(w, 5r)} \int_{B(w, 5\sigma^{-1}r)} d\Gamma(g, g) + V(w, r)g^2_{B(z_w, 0.25r)} \label{eq: Apply PI Twice} \\
    & \preceq \sum_{w \in W} \Psi(r)\int_{B(w, \sigma^{-1}r)} d\Gamma(g, g)  + \Psi(r) \int_{B(w, 5\sigma^{-1}r)} d\Gamma(g, g) + \int_{B(z_w, 0.25r)} g^2d\mu \label{eq: Apply VD + CS} \\
    & \preceq \Psi(r)\int_{\mathcal{X}_{8\sigma^{-1}r}} d\Gamma(g, g) + \int_{\mathcal{X}} g^2 d\mu \label{eq: Inside Domain}
\end{align}
where in \eqref{eq: Apply PI Twice} the first term follows from Poincar\'e inequality, the second term from a version of the latter as in \eqref{eq: Revised Poincare} after noting that $B(z_w, 0.25r) \subset B(w, 5r) \subset B(z_w, 10r)$ and applying volume doubling, and in \eqref{eq: Apply VD + CS} we have applied Cauchy-Schwarz and volume doubling. Finally, \eqref{eq: Inside Domain} follows from the bounded overlap property of the families $\{B(w, \sigma^{-1}r)\}_{w \in W}$ and $\{B(w, 5\sigma^{-1}r)\}_{w \in W}$ (by volume doubling) and the fact that $\{B(z_w, 0.25r)\}_{w \in W}$ are disjoint with only finitely many (independent of $r$) summands in the third term of \eqref{eq: Apply VD + CS} repeated (by the previous discussion). Combining \eqref{eq: Inside Domain} with \eqref{eq: Energy + L2 with Overlap} and noting $\Psi(r) \asymp \Psi\Big(\gamma^{\frac{1}{\beta - \alpha}}\Big) \preceq \gamma^{\frac{\beta}{\beta - \alpha}}$ we obtain our result. 
\end{proof} 

\section{Proof of Theorem \ref{Main Theorem}}
\label{Main Proof}
Throughout this section and the results in Section \ref{Auxiliary Estimates}, we assume $\mathcal{X}^{\circ}$ is $A$-uniform and the validity of \eqref{eq: VD}, \eqref{eq: RVD}, \eqref{eq: VD/RVD}, \eqref{eq: Poincare Inequality}, \eqref{eq: Exit Upper}, \eqref{eq: Exit Lower}, and \eqref{eq: Growth Condition}. 
\label{Error Decomposition}

\subsection{Bounding the Bias}
Bounding the first term, the bias, is straightforward. Indeed from some straightforward spectral calculus, we can see:
\begin{align}
    ||f^{*} - f_{\lambda}||^2_{L^2(\nu)} & = \lambda^2\|(C_{\nu} + \lambda)^{-1}f^{*}\|^2_{L^2(\nu)} \nonumber \\
    & = \lambda^2 \mathcal{E}(C^{\frac{1}{2}}_{\nu}(C_{\nu} + \lambda)^{-1}f^{*}, C^{\frac{1}{2}}_{\nu}(C_{\nu} + \lambda)^{-1}f^{*}) \nonumber \\
    & \leq \lambda^2 ||C^{\frac{1}{2}}_{\nu}(C_{\nu} + \lambda)^{-1}||^2_{\mathcal{H} \to \mathcal{H}} ||f^{*}||^2_{\mathcal{H}} \nonumber \\
    & \leq \lambda||f^{*}||^2_{\mathcal{H}} \label{eq: Bias Bound}
\end{align}

\subsection{Bounding the Variance}
We will now focus on the variance, which involves a large deviation estimate:

\begin{align}
    ||\hat{f}_{D, \lambda} - \hat{f}_{\lambda}||_{L^2(\nu)} & = \mathcal{E}(\hat{f}_{D, \lambda} - \hat{f}_{\lambda}, C_{\nu}(\hat{f}_{D, \lambda} - \hat{f}_{\lambda})) \label{eq: Apply Reproducing Covariance} \\
    & = ||C^{\frac{1}{2}}_{\nu}(\hat{f}_{D, \lambda} - \hat{f}_{\lambda})||_{\mathcal{H}} \nonumber \\
    & \leq \|C^{\frac{1}{2}}_{\nu}(C_{\nu} + \lambda)^{-\frac{1}{2}}\|_{\mathcal{H}} \|(C_{\nu} + \lambda)^{\frac{1}{2}}(\hat{C}_{D} + \lambda)^{-1}(C_{\nu} + \lambda)^{\frac{1}{2}}\|_{\mathcal{H} \to \mathcal{H}} \nonumber \\
    & \cdot \Big\|(C_{\nu} + \lambda)^{-\frac{1}{2}}\Big(\frac{\gamma_n}{n}\sum_{i = 1}^n Y_i e_{i, n} - (\hat{C}_{D} + \lambda)\hat{f}_{\lambda}\Big)\Big\|_{\mathcal{H}} \label{eq: Key Factor}
\end{align}
where \eqref{eq: Apply Reproducing Covariance} follows from \eqref{eq: Covariance Reproducing} with $m = \nu$. We may further break down the last factor in \eqref{eq: Key Factor} as:
\begin{align}
    (C_{\nu} + \lambda)^{-\frac{1}{2}}\Big(\frac{\gamma_n}{n}\sum_{i = 1}^n Y_i e_{i, n} - (\hat{C}_{D} + \lambda)\hat{f}_{\lambda}\Big) & = (C_{\nu} + \lambda)^{-\frac{1}{2}}\Big(\frac{\gamma_n}{n}\sum_{i = 1}^n Y_i e_{i, n} - (\hat{C}_{D} - \hat{C}_{\nu} + \hat{C}_{\nu} + \lambda)\hat{f}_{\lambda}\Big) \nonumber \\
    & = (C_{\nu} + \lambda)^{-\frac{1}{2}}\Big(\frac{\gamma_n}{n}\sum_{i = 1}^n Y_i e_{i, n} - (\hat{C}_{D}\hat{f}_{\lambda} - \hat{C}_{\nu}\hat{f}_{\lambda}) - \hat{C}_{\nu}f^{*}\Big) \nonumber \\
    & = (C_{\nu} + \lambda)^{-\frac{1}{2}}\Big(\Big(\frac{\gamma_n}{n}\sum_{i = 1}^n Y_i e_{i, n} - \gamma_n\int_{\mathcal{X}} e_{y, n}f^{*}(y)d\nu(y)\Big)  \nonumber \\ 
 & \hspace{2mm} - (\hat{C}_{D}\hat{f}_{\lambda} - \hat{C}_{\nu}\hat{f}_{\lambda}) + \Big(\int_{\mathcal{X}} \gamma_n e_{y, n}(f^{*}(y) - P^{\gamma_n}_{y}f^{*})d\nu(y)\Big)\Big) \label{eq: Final Breakdown}
\end{align}
where in \eqref{eq: Final Breakdown} we have grouped the terms into mean zero deviations and recalled the notation \eqref{eq: Projection Definition} from Theorem \ref{Capacitary Poincare}. \\

We bound the second factor in \eqref{eq: Key Factor} in section \ref{Operator Concentration Auxiliary} and the third factor \eqref{eq: Final Breakdown} in section \ref{Vector Concentration}


\subsubsection{Operator Concentration}
\label{Operator Concentration Auxiliary}
In this section, we will bound  the operator norm $\|(C_{\nu} + \lambda)^{\frac{1}{2}}(\hat{C}_{D} + \lambda)^{-1}(C_{\nu} + \lambda)^{\frac{1}{2}}\|_{\mathcal{H} \to \mathcal{H}}$ appearing in \eqref{eq: Key Factor}. First we write, similar to \cite{fischer2020sobolev}:
\begin{align}
    (C_{\nu} + \lambda)^{\frac{1}{2}}(\hat{C}_{D} + \lambda)^{-1}(C_{\nu} + \lambda)^{\frac{1}{2}} & = (C_{\nu} + \lambda)^{\frac{1}{2}}(\hat{C}_{D} - \hat{C}_{\nu} + \hat{C}_{\nu} -C_{\nu} + C_{\nu} + \lambda)^{-1}(C_{\nu} + \lambda)^{\frac{1}{2}} \nonumber \\
    & = (I - (C_{\nu} + \lambda)^{-\frac{1}{2}}(\hat{C}_{\nu} - \hat{C}_D)(C_{\nu} + \lambda)^{-\frac{1}{2}} \nonumber \\
    & - (C_{\nu} + \lambda)^{-\frac{1}{2}}(C_{\nu} - \hat{C}_{\nu})(C_{\nu} + \lambda)^{-\frac{1}{2}})^{-1} \label{eq: Final Operator Breakdown}
\end{align}
Note, we have derived an estimate for the $\lambda$-regularized operator ``renormalization error'' $\|(C_{\nu} + \lambda)^{-\frac{1}{2}}(C_{\nu} - \hat{C}_{\nu})(C_{\nu} + \lambda)^{-\frac{1}{2}}\|_{\mathcal{H} \to \mathcal{H}}$ in Lemma \ref{Operator Approximation Error}. Hence, in the following lemma, we will focus on estimating the operator ``sampling'' error $\|(C_{\nu} + \lambda)^{\frac{1}{2}}(\hat{C}_{\nu} - \hat{C}_D)(C_{\nu} + \lambda)^{\frac{1}{2}}\|_{\mathcal{H} \to \mathcal{H}}$. 

\begin{lemma}
\label{Operator Concentration}
Suppose $\gamma_n^{\frac{\beta}{\beta - \alpha}} \leq \frac{\lambda}{5K}$ (where $K$ is the constant in Lemma \ref{Operator Approximation Error}) and $\lambda \leq \|C_{\nu}\|_{\mathcal{H} \to \mathcal{H}}$. Then, there exists a $C > 1$, such that:
\begin{equation*}
    \|(C_{\nu} + \lambda)^{\frac{1}{2}}(\hat{C}_{\nu} - \hat{C}_D)(C_{\nu} + \lambda)^{\frac{1}{2}}\|_{\mathcal{H} \to \mathcal{H}} \leq \frac{4C\log \Big(\frac{1}{\delta}\Big) \gamma_n \log \Big(\frac{4e\gamma_n}{\lambda}\Big)}{3n\lambda} + \sqrt{\frac{2C^2\log\Big(\frac{1}{\delta}\Big) \gamma_n \log \Big(\frac{4e\gamma_n}{\lambda}\Big)}{n\lambda}}
\end{equation*}
with probability $1 - \delta$
\end{lemma}
\begin{proof}
Our proof is inspired by the concentration technique in Lemma 17 of \cite{fischer2020sobolev}. Define the random element of $\mathcal{H}$:
\begin{equation*}
    \xi_{x} = \gamma_n(C_{\nu} + \lambda)^{-\frac{1}{2}}e_{x, n} 
\end{equation*}
with $x \sim \nu$. Then, we have:
\begin{align}
    ||\xi_{x} \otimes \xi_{x}||_{\mathcal{H} \to \mathcal{H}} & \leq \frac{\gamma_n}{\lambda} \label{eq: Operator Norm Bound} \\
    \mathbb{E}_{\nu}[||\xi_{x} \otimes \xi_{x}||^2_{\mathcal{H} \to \mathcal{H}}] & = \mathbb{E}_{\nu}[||\xi_{x}||^2_{\mathcal{H}}\xi_{x} \otimes \xi_{x}] \nonumber \\
    & \leq \frac{\gamma_n}{\lambda} \mathbb{E}_{\nu}[\xi_{x} \otimes \xi_{x}] \equiv V \nonumber   
\end{align}
where in \eqref{eq: Operator Norm Bound} we have applied Lemma \ref{Regularized Potential Norm}. For any $h \in \mathcal{H}$, let $g = (C_{\nu} + \lambda)^{-\frac{1}{2}}h$ and observe $g \in L^2(\nu)$ from \eqref{eq: Covariance Reproducing} with $m = \nu$. Then, from Lemma \ref{Average Projection}, we have:
\begin{align}
    \mathbb{E}_{\nu}[\langle h, \xi_{x} \rangle^2_{\mathcal{H}}] & = \int (P_{x}^{\gamma_n} g)^2 d\nu(x) \nonumber \\
    & \preceq \gamma_n^{\frac{\beta}{\beta - \alpha}}\mathcal{E}(g, g) + ||g||^2_{L^2(\nu)} \nonumber \\
    & \leq \frac{\gamma_n^{\frac{\beta}{\beta - \alpha}}}{\lambda}\mathcal{E}(h, h) + \mathcal{E}(h, h) \label{eq: Definition of g} \\
    & \leq \Big(\frac{1}{5K} +  1\Big)\mathcal{E}(h, h) \label{eq: Apply Ratio}
\end{align}
where \eqref{eq: Definition of g} follows the definition of $g$ and the fact that $||g||^2_{L^2(\nu)} = \mathcal{E}(C^{\frac{1}{2}}_{\nu}g, C^{\frac{1}{2}}_{\nu}g)$, and \eqref{eq: Apply Ratio} follows from the assumption that $\frac{\gamma_n^{\frac{\beta}{\beta - \alpha}}}{\lambda} \leq \frac{1}{5K}$. Likewise, we have that:
\begin{align*}
    \frac{\lambda}{\gamma_n}V & = (C_{\nu} + \lambda)^{-\frac{1}{2}}\hat{C}_{\nu}(C_{\nu} + \lambda)^{-\frac{1}{2}} \\
    & = (C_{\nu} + \lambda)^{-\frac{1}{2}}C_{\nu}(C_{\nu} + \lambda)^{-\frac{1}{2}} + (C_{\nu} + \lambda)^{-\frac{1}{2}}(\hat{C}_{\nu} - C_{\nu})(C_{\nu} + \lambda)^{-\frac{1}{2}}
\end{align*}
Thus, by Lemma \ref{Operator Approximation Error}, we have that:
\begin{align}
    ||\lambda \gamma^{-1}_n V||_{\mathcal{H} \to \mathcal{H}} & \geq \|(C_{\nu} + \lambda)^{-\frac{1}{2}}C_{\nu}(C_{\nu} + \lambda)^{-\frac{1}{2}}\|_{\mathcal{H} \to \mathcal{H}} - \|(C_{\nu} + \lambda)^{-\frac{1}{2}}(\hat{C}_{\nu} - C_{\nu})(C_{\nu} + \lambda)^{-\frac{1}{2}}\|_{\mathcal{H} \to \mathcal{H}} \nonumber \\
    & \geq \frac{1}{2} - \sqrt{\frac{K^2\gamma^{\frac{2\beta}{\beta - \alpha}}_n}{\lambda^2} + \frac{K\gamma^{\frac{\beta}{\beta - \alpha}}_n}{\lambda}} \label{eq: Apply Operator Lower Bound} \\ 
    & \geq \frac{1}{2} - \frac{\sqrt{6}}{5}  \label{eq: Plug in Ratio Bound}
\end{align}
where in \eqref{eq: Apply Operator Lower Bound} we have applied Lemma \ref{Operator Approximation Error} and $\lambda \leq \|C_{\nu}\|_{\mathcal{H} \to \mathcal{H}}$, and in \eqref{eq: Plug in Ratio Bound} the assumption that $\frac{\gamma_n^{\frac{\beta}{\beta - \alpha}}}{\lambda} \leq \frac{1}{5K}$. Hence, there exists a $\tilde{K} > 0$ such that:
\begin{align}
   \frac{\gamma_n}{\tilde{K}\lambda} \leq ||V||_{\mathcal{H} \to \mathcal{H}} & \leq \frac{\tilde{K}\gamma_n}{\lambda} \label{eq: V Norm Bound} \\
    \text{tr}(V) & \leq \frac{\gamma^2_n}{\lambda^2} \label{eq: V Trace Bound}
\end{align}
where \eqref{eq: V Trace Bound} follows from \eqref{eq: Operator Norm Bound} and definition. Now, observing that:
\begin{equation*}
    (C_{\nu} + \lambda)^{-\frac{1}{2}}(\hat{C}_{D} - \hat{C}_{\nu})(C_{\nu} + \lambda)^{-\frac{1}{2}}  = \frac{1}{n} \sum_{i = 1}^n \xi_{X_i} \otimes \xi_{X_i} - \mathbb{E}_{\nu}[\xi_x \otimes \xi_x]
\end{equation*}
and applying Bernstein's inequality (Theorem 27 of \cite{fischer2020sobolev}; see also Lemma 26 of \cite{lin2020optimal}), we obtain our result. 
\end{proof}

\begin{lemma}
\label{Operator Concentration Main}
Suppose  $\gamma_n^{\frac{\beta}{\beta - \alpha}} \leq \frac{\lambda}{5K}$ and additionally that $n \geq \frac{32C^2\log \Big(\frac{1}{\delta}\Big) \gamma_n \log \Big(\frac{4e\gamma_n}{\lambda}\Big)}{ \lambda}$ (where $K$ is the constant from Lemma \ref{Operator Approximation Error} and $C$ is the constant from Lemma \ref{Operator Concentration}). Then:
\begin{equation*}
       \|(C_{\nu} + \lambda)^{\frac{1}{2}}(\hat{C}_{D} + \lambda)^{-1}(C_{\nu} + \lambda)^{\frac{1}{2}}\|_{\mathcal{H} \to \mathcal{H}} \preceq 1
\end{equation*}
with probability $1 - \delta$. 
\end{lemma}
\begin{proof}
From \eqref{eq: Final Operator Breakdown} and a Neumann series expansion we have that:
\begin{align}
    \|(C_{\nu} + \lambda)^{\frac{1}{2}}(\hat{C}_{D} + \lambda)^{-1}(C_{\nu} + \lambda)^{\frac{1}{2}}\|_{\mathcal{H} \to \mathcal{H}} & = \|(I - (C_{\nu} + \lambda)^{-\frac{1}{2}}(\hat{C}_{\nu} - \hat{C}_D)(C_{\nu} + \lambda)^{-\frac{1}{2}} \nonumber \\
    & - (C_{\nu} + \lambda)^{-\frac{1}{2}}(C_{\nu} - \hat{C}_{\nu})(C_{\nu} + \lambda)^{-\frac{1}{2}})^{-1}\|_{\mathcal{H} \to \mathcal{H}} \nonumber \\
    & \leq \sum_{i = 0}^{\infty} 2^{i - 1}(\|(C_{\nu} + \lambda)^{-\frac{1}{2}}(\hat{C}_{\nu} - \hat{C}_D)(C_{\nu} + \lambda)^{-\frac{1}{2}}\|^i_{\mathcal{H} \to \mathcal{H}}  \nonumber \\
    & \hspace{2mm} + \|(C_{\nu} + \lambda)^{-\frac{1}{2}}(C_{\nu} - \hat{C}_{\nu})(C_{\nu} + \lambda)^{-\frac{1}{2}}\|^i_{\mathcal{H} \to \mathcal{H}}) \nonumber \\
    & \preceq \sum_{i = 0}^{\infty} 2^{i - 1}\Big(\Big(\frac{7}{24}\Big)^i + \Big(\frac{\sqrt{6}}{5}\Big)^i\Big) \label{eq: Apply Operator Concentration + Approximation}\\
    & \preceq 1 \nonumber
\end{align}
where \eqref{eq: Apply Operator Concentration + Approximation} follows from Lemmas \ref{Operator Concentration} and \ref{Operator Approximation Error} and the assumptions that $\gamma_n^{\frac{\beta}{\beta - \alpha}} \leq \frac{\lambda}{5K}$ and $n \geq \frac{32C^2\log \Big(\frac{1}{\delta}\Big) \gamma_n \log \Big(\frac{4e\gamma_n}{\lambda}\Big)}{\lambda}$
\end{proof}
\subsubsection{Function Concentration}
\label{Vector Concentration}
We now focus on estimating \eqref{eq: Final Breakdown}. We first note that the last term in \eqref{eq: Final Breakdown}

\begin{equation*}
    \int_{\mathcal{X}} \gamma_n e_{y, n}(f^{*}(y) - P^{\gamma_n}_{y}f^{*})d\nu(y)
\end{equation*}
is nonrandom and will be bounded by Lemma \ref{Average Regularized Projection}. We will estimate the remaining two terms using Markov's inequality (note that tighter estimates, using, for example, Bernstein's inequality, are not available, due to the lack of an embedding of $\mathcal{H}$ in $L^p(\nu)$ for $p > \frac{2\alpha}{\alpha - \beta}$). 

\begin{lemma}
\label{Vector Concentration Main}
Suppose $\gamma_n^{\frac{\beta}{\beta - \alpha}} \leq 0.5 K^{-1}\lambda$ (where $K$ is from Lemma \ref{Operator Approximation Error}). Then, there exists a $\tilde{C} > 0$ (independent of $\lambda$ and $n$) such that:
\begin{align*}
   \Big\|(C_{\nu} + \lambda)^{-\frac{1}{2}}\Big(\Big(\frac{\gamma_n}{n}\sum_{i = 1}^n Y_i e_{i, n} - \int_{\mathcal{X}} \gamma_n e_{y, n}f^{*}(y)d\nu(y)\Big) 
 - (\hat{C}_{D}\hat{f}_{\lambda} - \hat{C}_{\nu}\hat{f}_{\lambda})\Big)\Big\|^2_{\mathcal{H}} & \leq \frac{\tilde{C}\gamma_n}{\delta \lambda n} \Big(\rho^2 + (\lambda  + \gamma_n^{\frac{\beta}{\beta - \alpha}} )\|f^{*}\|^2_{\mathcal{H}}\Big) 
\end{align*}
with probability $1 - \delta$. 
\end{lemma}
\begin{proof}
Let:
\begin{equation*}
    G_n \equiv \frac{\gamma_n}{n}\sum_{i = 1}^n (Y_i - P_{X_i}^{\gamma_n}\hat{f}_{\lambda}) (C_{\nu} + \lambda)^{-\frac{1}{2}} e_{i, n}
\end{equation*}
Then, from definitions \eqref{eq: Cutoff Covariance Kernel} and \eqref{eq: Sample Cutoff Covariance Kernel}, we may write:
\begin{align*}
    (C_{\nu} + \lambda)^{-\frac{1}{2}}\Big(\Big(\frac{1}{n}\sum_{i = 1}^n \gamma_n Y_i e_{i, n} - \int_{\mathcal{X}} \gamma_n e_{y, n}f^{*}(y)d\nu(y)\Big) 
 - (\hat{C}_{D}\hat{f}_{\lambda} - \hat{C}_{\nu}\hat{f}_{\lambda})\Big) & = G_n - \mathbb{E}_P[G_n] 
\end{align*}
We have:
\begin{align}
    \mathbb{E}_{\mathbb{P}}[\|G_n - \mathbb{E}_P[G_n] \|^2_{\mathcal{H}}]  &  \leq \frac{\gamma^2_n}{n}\mathbb{E}_{\mathbb{P}}[\|(C_{\nu} + \lambda)^{-\frac{1}{2}}e_{x, 
 n}\|^2_{\mathcal{H}}(y - P^{\gamma_n}_{x}\hat{f}_{\lambda})^2] \nonumber  \\
 & = \frac{\gamma^2_n}{n}\mathbb{E}_{\mathbb{P}}[\|(C_{\nu} + \lambda)^{-\frac{1}{2}}e_{x, n}\|^2_{\mathcal{H}}(y - f^{*}(x) + f^{*}(x) -\hat{f}_{\lambda}(x) \nonumber \\
 &  \hspace{2mm} - P^{\gamma_n}_{x}\hat{f}_{\lambda} + \hat{f}_{\lambda}(x))^2] \nonumber \\ 
    & \leq  \frac{3\gamma^2_n}{n}\mathbb{E}_{\mathbb{P}}[\|(C_{\nu} + \lambda)^{-\frac{1}{2}}e_{x, n}\|^2_{\mathcal{H}}((y - f^{*}(x))^2 + (f^{*}(x) -\hat{f}_{\lambda}(x))^2 \nonumber \\
    & \hspace{2mm} + (P^{\gamma_n}_{x}\hat{f}_{\lambda} - \hat{f}_{\lambda}(x))^2)] \nonumber \\
    & \leq \frac{3\gamma_n}{n\lambda}\mathbb{E}_{\mathbb{P}}[(y - f^{*}(x))^2 + (f^{*}(x) -\hat{f}_{\lambda}(x))^2 + (P^{\gamma_n}_{x}\hat{f}_{\lambda} - \hat{f}_{\lambda}(x))^2] \label{eq: Variance Decomposition}
\end{align}
where \eqref{eq: Variance Decomposition} follows from Lemma \ref{Regularized Potential Norm}. We first focus on bounding the second term in \eqref{eq: Variance Decomposition}. We observe that:
\begin{equation*}
    f^{*} - \hat{f}_{\lambda} = \lambda (\hat{C}_{\nu} + \lambda)^{-1} f^{*} 
\end{equation*}
by definition \eqref{eq: Regularized Cutoff}. Hence, by definition, we have that:
\begin{align}
    \mathbb{E}_{\mathbb{P}}[(f^{*}(x) - \hat{f}_{\lambda}(x))^2] & = \lambda^2 \mathcal{E}(C^{\frac{1}{2}}_{\nu}(\hat{C}_{\nu} + \lambda)^{-1} f^{*}, C_{\nu}^{\frac{1}{2}}(\hat{C}_{\nu} + \lambda)^{-1} f^{*}) \nonumber \\
    & = \lambda^2 \mathcal{E}(C^{\frac{1}{2}}_{\nu}(C_{\nu} + \lambda)^{-\frac{1}{2}}(C_{\nu} + \lambda)^{\frac{1}{2}}(\hat{C}_{\nu} + \lambda)^{-1} f^{*}, C^{\frac{1}{2}}_{\nu}(C_{\nu} + \lambda)^{-\frac{1}{2}}(C_{\nu} + \lambda)^{\frac{1}{2}}(\hat{C}_{\nu} + \lambda)^{-1} f^{*}) \nonumber \\
    & \leq \lambda^2 \|(C_{\nu} + \lambda)^{\frac{1}{2}}(\hat{C}_{\nu} + \lambda)^{-1}\|^2_{\mathcal{H} \to \mathcal{H}} \mathcal{E}(f^{*}, f^{*}) \label{eq: Population Error with Projection Average}
\end{align}
We estimate $\|(C_{\nu} + \lambda)^{\frac{1}{2}}(\hat{C}_{\nu} + \lambda)^{-1}\|_{\mathcal{H} \to \mathcal{H}}$, using Lemma \ref{Operator Approximation Error}: 
\begin{align}
    \|(C_{\nu} + \lambda)^{\frac{1}{2}}(\hat{C}_{\nu} + \lambda)^{-1}\|_{\mathcal{H} \to \mathcal{H}} & = \|(I - (C_{\nu} + \lambda)^{-\frac{1}{2}}(C_{\nu} - \hat{C}_{\nu})(C_{\nu} + \lambda)^{-\frac{1}{2}})^{-1}(C_{\nu} + \lambda)^{-\frac{1}{2}}\|_{\mathcal{H} \to \mathcal{H}} \nonumber \\
    & \leq \Big(\sum_{i = 0}^{\infty} \|(C_{\nu} + \lambda)^{-\frac{1}{2}}(C_{\nu} - \hat{C}_{\nu})(C_{\nu} + \lambda)^{-\frac{1}{2}}\|_{\mathcal{H} \to \mathcal{H}}^i \Big)\|(C_{\nu} + \lambda)^{-\frac{1}{2}}\|_{\mathcal{H} \to \mathcal{H}} \label{eq: Neumann Series} \\
    & \preceq \frac{1}{\sqrt{\lambda}}\sum_{i = 0}^{\infty} \Big(\frac{K^2\gamma_n^{\frac{2\beta}{\beta - \alpha}}}{\lambda^2} + \frac{K\gamma_n^{\frac{\beta}{\beta - \alpha}}}{\lambda}\Big)^{\frac{i}{2}} \label{eq: Apply Operator Approximation Error} \\
    & \leq \frac{1}{\sqrt{\lambda}}\Big(1 - \frac{\sqrt{3}}{2} \Big)^{-1} \label{eq: Ratio Decay Assumption}
\end{align}
where \eqref{eq: Neumann Series} follows from a Neumann series expansion, \eqref{eq: Apply Operator Approximation Error} follows from Lemma \ref{Operator Approximation Error}, and \eqref{eq: Ratio Decay Assumption} follows from the assumption $\frac{\gamma_n^{\frac{\beta}{\beta - \alpha}}}{\lambda} \leq 0.5K^{-1}$. Substituting  back into \eqref{eq: Population Error with Projection Average}, we obtain:
\begin{equation}
\label{eq: Second Term Bound}
 \mathbb{E}_{\mathbb{P}}[(f^{*}(x) - \hat{f}_{\lambda}(x))^2] \preceq \lambda \mathcal{E}(f^{*}, f^{*})
\end{equation}
Now, we bound the third term in \eqref{eq: Variance Decomposition} using Lemma \ref{Capacitary Poincare}. Indeed, we have that:
\begin{equation}
\label{eq: Third Term Bound}
    \mathbb{E}_{\mathbb{P}}[(P^{\gamma_n}_{x}\hat{f}_{\lambda} - \hat{f}_{\lambda}(x))^2] \preceq \gamma_n^{\frac{\beta}{\beta - \alpha}} \mathcal{E}(\hat{f}_{\lambda}, \hat{f}_{\lambda}) \leq \gamma_n^{\frac{\beta}{\beta - \alpha}} \mathcal{E}(f^{*}, f^{*})
\end{equation}
by definition of $\hat{f}_{\lambda}$ in \eqref{eq: Regularized Cutoff}. Finally, the first term is simply equal to the noise variance:
\begin{equation}
\label{eq: Noise Term}
    \mathbb{E}_{\mathbb{P}}[(y - f^{*}(x))^2] = \rho^2
\end{equation}
Combining \eqref{eq: Noise Term} with \eqref{eq: Second Term Bound} and \eqref{eq: Third Term Bound}, and recalling $\|f^{*}\|^2_{\mathcal{H}} = \mathcal{E}(f^{*}, f^{*})$ from \eqref{eq: 0-norm}, we obtain our result from Markov's inequality. 
\end{proof}

\subsection{Bounding the Approximation Error}

In this section we estimate the approximation error $||\hat{f}_{\lambda} - f_{\lambda}||_{L^2(\nu)}$ between the regularized population estimators, incurred by the approximation of pointwise values by their capacitary means. Unsuprisingly, Theorem \ref{Capacitary Poincare} and Lemma \ref{Average Projection}, along with the auxiliary results of section \ref{Auxiliary Estimates} will be pivotal. We may decompose $||\hat{f}_{\lambda} - f_{\lambda}||_{L^2(\nu)}$, similarly to \eqref{eq: Key Factor}. Namely, we have that:

\begin{align}
    ||\hat{f}_{\lambda} - f_{\lambda}||_{L^2(\nu)} & \leq \|C^{\frac{1}{2}}_{\nu}(C_{\nu} + \lambda)^{-\frac{1}{2}}\|_{\mathcal{H} \to \mathcal{H}} \|(C_{\nu} + \lambda)^{\frac{1}{2}}(\hat{C}_{\nu} + \lambda)^{-1}(C_{\nu} + \lambda)^{\frac{1}{2}}\|_{\mathcal{H} \to \mathcal{H}} \cdot \nonumber \\
    & \|(C_{\nu} + \lambda)^{-\frac{1}{2}}(\hat{C}_{\nu}f^{*} - (\hat{C}_{\nu} + \lambda)f_{\lambda})\|_{\mathcal{H}} \nonumber \\ 
    & \leq  \|(C_{\nu} + \lambda)^{\frac{1}{2}}(\hat{C}_{\nu} + \lambda)^{-1}(C_{\nu} + \lambda)^{\frac{1}{2}}\|_{\mathcal{H} \to \mathcal{H}}  \|(C_{\nu} + \lambda)^{-\frac{1}{2}}(\hat{C}_{\nu}f^{*} - C_{\nu}f^{*} + (C_{\nu}-\hat{C}_{\nu})f_{\lambda})\|_{\mathcal{H}} \label{eq: Approximation Breakdown}
\end{align}

\begin{lemma}
\label{Approximation Neumann}
Suppose $\gamma_n^{\frac{\beta}{\beta - \alpha}} \leq 0.5 K^{-1}\lambda$ (where $K$ is the constant from Lemma \ref{Operator Approximation Error}). Then:
\begin{equation*}
    \|(C_{\nu} + \lambda)^{\frac{1}{2}}(\hat{C}_{\nu} + \lambda)^{-1}(C_{\nu} + \lambda)^{\frac{1}{2}}\|_{\mathcal{H} \to \mathcal{H}} \preceq 1
\end{equation*}
\end{lemma}

\begin{proof}
Like in section \ref{Operator Concentration Auxiliary}, we estimate $\|(C_{\nu} + \lambda)^{\frac{1}{2}}(\hat{C}_{\nu} + \lambda)^{-1}(C_{\nu} + \lambda)^{\frac{1}{2}}\|_{\mathcal{H} \to \mathcal{H}}$ using a Neumannn series expansion:
\begin{align}
    \|(C_{\nu} + \lambda)^{\frac{1}{2}}(\hat{C}_{\nu} + \lambda)^{-1}(C_{\nu} + \lambda)^{\frac{1}{2}}\|_{\mathcal{H} \to \mathcal{H}} & = \|(I - (C_{\nu} + \lambda)^{-\frac{1}{2}}(C_{\nu} - \hat{C}_{\nu})(C_{\nu} + \lambda)^{-\frac{1}{2}})^{-1}\|_{\mathcal{H} \to \mathcal{H}} \nonumber \\
    & \leq \sum_{i = 0}^{\infty} \|(C_{\nu} + \lambda)^{-\frac{1}{2}}(C_{\nu} - \hat{C}_{\nu})(C_{\nu} + \lambda)^{-\frac{1}{2}}\|^i_{\mathcal{H} \to \mathcal{H}} \nonumber \\
    & \leq \Big(1 - \sqrt{\frac{3}{4}}\Big)^{-1} \label{eq: Apply Operator Approximation 2}
\end{align}
where \eqref{eq: Apply Operator Approximation 2} follows from Lemma \ref{Operator Approximation Error} and the assumption that $\gamma_n^{\frac{\beta}{\beta - \alpha}} \leq 0.5 K^{-1}\lambda$
\end{proof}

\begin{lemma}
\label{Approximation Vector}
Suppose $\gamma_n^{\frac{\beta}{\beta - \alpha}} \leq 0.5 K^{-1}\lambda$ (where $K$ is the constant from Lemma \ref{Operator Approximation Error})
\begin{equation*}
     \|(C_{\nu} + \lambda)^{-\frac{1}{2}}(\hat{C}_{\nu}f^{*} - C_{\nu}f^{*} + (C_{\nu}-\hat{C}_{\nu})f_{\lambda})\|_{\mathcal{H}} \preceq \sqrt{\gamma_n^{\frac{\beta}{\beta - \alpha}}} ||f^{*}||_{\mathcal{H}} 
\end{equation*}
\end{lemma}
\begin{proof}
We have that:
\begin{align}
    ||f^{*} - f_{\lambda}||_{\mathcal{H}} & = \|\lambda(C_{\nu} + \lambda)^{-1}f^{*}\|_{\mathcal{H}} \leq \|f^{*}\|_{\mathcal{H}} \nonumber \\
    ||f^{*} - f_{\lambda}||_{L^2(\nu)} & = \|\lambda C^{\frac{1}{2}}_{\nu}(C_{\nu} + \lambda)^{-1}f^{*}\|_{\mathcal{H}} \leq \sqrt{\lambda \|f^{*}\|^2_{\mathcal{H}}} \label{eq: Apply Green Reproducing}
\end{align}
where in \eqref{eq: Apply Green Reproducing} we have applied \eqref{eq: Bias Bound}. The claim then follows from Lemma \ref{Sampling Approximation Error} and the assumption that $\gamma_n^{\frac{\beta}{\beta - \alpha}} \leq 0.5 K^{-1}\lambda$. 
\end{proof}

\subsection{Finishing the Proof}

We first note that with the prescribed choice of $\lambda \asymp n^{-\frac{\beta}{\alpha + \beta}}$ and $\gamma_n \asymp n^{\frac{\alpha - \beta}{\alpha + \beta}}$, we have that, as $n \to \infty$:
\begin{align}
    \gamma_n^{\frac{\beta}{\beta - \alpha}} & \asymp n^{-\frac{\beta}{\alpha + \beta}} \asymp \lambda \label{eq: Escape-to-Reg Ratio} \\
    \frac{\gamma_n}{\lambda}  & \asymp \frac{n^{\frac{\alpha - \beta}{\alpha + \beta}}}{n^{-\frac{\beta}{\alpha + \beta}}} \asymp n^{\frac{\alpha}{\alpha + \beta}} \preceq n \label{eq: Cutoff-to-Reg Ratio}
\end{align}
Thus, with an appropriate choice of proportionality constant (e.g. $\lambda = 6K\gamma_n^{\frac{\beta}{\beta - \alpha}}$), we have that the conditions of Lemmas \ref{Operator Concentration}, \ref{Operator Concentration Main}, \ref{Vector Concentration Main}, \ref{Approximation Neumann}, \ref{Approximation Vector} are satisfied for sufficiently large $n \geq 1$. Hence, combining these results with \eqref{eq: Bias Bound}, \eqref{eq: Apply Reproducing Covariance}, \eqref{eq: Key Factor}, \eqref{eq: Final Breakdown}, Lemma \ref{Average Regularized Projection}, and \eqref{eq: Approximation Breakdown} we obtain, with probability $1 - 3\delta$:

\begin{align}
    ||\hat{f}_{D, \lambda} - f^{*}||^2_{L^2(\nu)} & \preceq \lambda ||f^{*}||^2_{\mathcal{H}} + \frac{\tilde{C}\gamma_n}{\delta \lambda n} \Big(\rho^2 + \lambda ||f^{*}||^2_{\mathcal{H}} + \gamma_n^{\frac{\beta}{\beta - \alpha}} ||f^{*}||^2_{\mathcal{H}}\Big) \nonumber \\
    & + ||f^{*}||^2_{\mathcal{H}}\Bigg(\frac{\gamma_n^{\frac{2\beta}{\beta - \alpha}}}{\lambda} + \gamma_n^{\frac{\beta}{\beta - \alpha}}\Bigg) + ||f^{*}||^2_{\mathcal{H}}\gamma_n^{\frac{\beta}{\beta - \alpha}} \nonumber \\
    & \preceq \Big(\lambda + \frac{\gamma_n}{\delta \lambda n} + \gamma_n^{\frac{\beta}{\beta - \alpha}}\Big)||f^{*}||^2_{\mathcal{H}} \label{eq: Balancing Breakdown}
\end{align}
where in \eqref{eq: Balancing Breakdown}, we have applied $\gamma_n^{\frac{\beta}{\beta - \alpha}} \asymp \lambda \to 0$ as $n \to \infty$ from \eqref{eq: Escape-to-Reg Ratio} and recalled $||f||^2_{\mathcal{H}} = \mathcal{E}(f, f)$ from \eqref{eq: 0-norm}. Substituting our choices of $\lambda \asymp n^{-\frac{\beta}{\alpha + \beta}}$ and $\gamma_n \asymp n^{\frac{\alpha - \beta}{\alpha + \beta}}$ we obtain our result. 

\section{Proof of Proposition \ref{Covering Numbers} and Theorem \ref{ERM Risk}}
\label{ERM Proofs}
\begin{lemma}
\label{Modulus of Continuity}
For $x, y \in \mathcal{X}$:
\begin{equation*}
    \gamma^2 \mathcal{E}(e^{\gamma}_{x} - e^{\gamma}_{y}, e^{\gamma}_{x} - e^{\gamma}_{y}) \preceq \gamma^{\frac{\alpha - 1}{\alpha - 2}} d(y, x)
\end{equation*}
\end{lemma}
\begin{proof}
Recall from Theorem 1.2 in \cite{grigor2015generalized} (see also Theorem 7.5 in \cite{grigor2014heat}) that, for any compact set $\mathcal{S} \subset \mathcal{M}$, there exists a constant $C > 1$ such that:
\begin{equation}
\label{eq: Green Bounds}
    C^{-1}d(x, y)^{2 - \alpha} \leq g(x, y) \leq C d(x, y)^{2 - \alpha}
\end{equation}
for $y \in \mathcal{M} \setminus \{x\}$ and $x \in \mathcal{S}$. Let $x, y \in \mathcal{X}$ with $d(x, y) \leq \frac{1}{4}(C^{-1}\gamma)^{\frac{1}{2 - \alpha}} \equiv r$, where $C > 1$ is from \eqref{eq: Green Bounds}. Then, we have that:
\begin{align}
    \gamma^2 \mathcal{E}(e^{\gamma}_{x}, e^{\gamma}_{y}) & = \gamma^2 \int e^{\gamma}_{x}(z) d\nu^{\gamma}_{y}(z) \label{eq: Eq Intro}\\
    & = \gamma -  \gamma^2 \int (1 - e^{\gamma}_{x}(z)) d\nu^{\gamma}_{y}(z)                             \label{eq: Eq Mass} \\
    & = \gamma - \gamma \int (g(z, y) \wedge \gamma - g(z, x) \wedge \gamma)d\nu^{\gamma}_{y}(z) \label{eq: Definition of Potential}\\
    & \geq \gamma - \|\nabla(g(z, \cdot))\|_{L^{\infty}(B(y, r))} d_{\mathcal{E}}(y, x) \label{eq: Intrinsic Definition} \\ 
    & \succeq \gamma - \|\nabla(g(z, \cdot))\|_{L^{\infty}(B(y, r))} d(y, x) \label{eq: Metric Equivalence}  \\
    & \succeq \gamma - \frac{\|g(z, \cdot)\|_{L^{\infty}(B(y, 2r))}d(y, x)}{r}  \label{eq: Apply Li-Yau}
\end{align}
where in \eqref{eq: Eq Intro}, $\nu^{\gamma}_y(z)$ is the equilibrium measure associated with the potential  $e^{\gamma}_y(z)$, \eqref{eq: Eq Mass} uses the fact that $\nu^{\gamma}_y(\mathcal{M}) = \gamma^{-1}$ by definition, \eqref{eq: Definition of Potential} applies the definition \eqref{Green Potential}, \eqref{eq: Intrinsic Definition} applies the definition of $d_{\mathcal{E}}$, \eqref{eq: Metric Equivalence} applies \eqref{Metrics are Equivalent}, and \eqref{eq: Apply Li-Yau} invokes \eqref{eq: Bakry-Emery} and applies the Li-Yau gradient estimate guaranteed by Theorem 1.2 in \cite{coulhon2020gradient}. We may apply the latter result as $g(z, \cdot)$ is harmonic in $B(y, 2r)$ as for any $w \in B(y, 2r)$, $d(w, z) \geq d(z, y) - d(y, w) \geq 4r - 2r = 2r$ by \eqref{eq: Green Bounds}. Moreover, we have that for $w \in B(y, 2r)$:
\begin{equation*}
    g(z, w) \leq C(2r)^{2 - \alpha} = C^2 2^{2- \alpha} \cdot C^{-1}r^{2 - \alpha} \preceq \gamma
\end{equation*}
again by \eqref{eq: Green Bounds}. Hence, we have that:
\begin{align*}
    \gamma^2 \mathcal{E}(e^{\gamma}_{x} - e^{\gamma}_{y}, e^{\gamma}_{x} - e^{\gamma}_{y}) & = \gamma^2(\mathcal{E}(e^{\gamma}_{x}, e^{\gamma}_{x}) - 2\mathcal{E}(e^{\gamma}_{x}, e^{\gamma}_{y}) + \mathcal{E}(e^{\gamma}_{y}, e^{\gamma}_{y})) \\
    & = 2\gamma - 2\gamma^2\mathcal{E}(e^{\gamma}_{x}, e^{\gamma}_{y}) \\
    & \preceq \frac{2\|g(z, \cdot)\|_{L^{\infty}(B(y, 2r))}d(y, x)}{r} \\
    & \preceq \frac{\gamma}{r} \cdot d(y, x)
\end{align*}
Moreover, when $d(x, y) > r$, we have from Cauchy-Schwarz:
\begin{align*}
    \gamma^2 \mathcal{E}(e^{\gamma}_{x} - e^{\gamma}_{y}, e^{\gamma}_{x} - e^{\gamma}_{y}) & \leq 2\gamma^2 (\mathcal{E}(e^{\gamma}_{x}, e^{\gamma}_{x})  + \mathcal{E}(e^{\gamma}_{y}, e^{\gamma}_{y})) \\
    & \preceq \gamma \\
    & \leq \frac{\gamma}{r} \cdot d(y, x)
\end{align*}
Substituting our choice for $r$, we obtain the result. 
\end{proof}

\begin{proof}[Proof of Proposition \ref{Covering Numbers}]
We apply the approach of \cite{kuelbs1993metric} to estimate the covering numbers of $\mathcal{H}_{\gamma}$ by estimating small ball probabilities for the Gaussian measure $\lambda_{\mathcal{H}}$ with Cameron-Martin space $\mathcal{H}$. Namely, we wish to estimate:
\begin{equation*}
    \lambda_{\mathcal{H}}(K_{\gamma}(\epsilon)) 
\end{equation*}
as $\epsilon \to 0$, where we recall the definition of $K_{\gamma}(\epsilon)$ in \eqref{eq: Regularized Sup Ball}. Let $h$ be a random vector with distribution $\lambda_{\mathcal{H}}$. We study the centered Gaussian field $\{h_{x}\}_{x \in \mathcal{X}}$ over $\mathcal{X}$ given by:
\begin{align*}
    h_x & = \langle h, \gamma e^{\gamma}_x \rangle \\
    \text{Cov}(h_x, h_y) & = \mathcal{E}(\gamma e^{\gamma}_x, \gamma e^{\gamma}_y)
\end{align*}
Note, here we suppress the dependence on $\gamma > 0$ in writing $\{h_{x}\}_{x \in \mathcal{X}}$ for notational simplicity. We define the following canonical distance associated with the field $\{h_{x}\}_{x \in \mathcal{X}}$:
\begin{equation*}
    d_{\gamma}(x, y) = \sqrt{\text{Cov}[(h_x - h_y)^2]} = \sqrt{\gamma^2 \mathcal{E}(e^{\gamma}_{x} - e^{\gamma}_{y}, e^{\gamma}_{x} - e^{\gamma}_{y})}
\end{equation*}
Since, we are only interested in asymptotic behavior, we assume w.l.o.g. $\text{diam}(\mathcal{X}, d_{\gamma}) = 1$ (this can always be achieved by rescaling $\gamma$ without affecting asymptotics). We begin by following the approach in \cite{talagrand1993new} (see p. 524 therein) to demonstrate:
\begin{equation*}
    \lambda_{\mathcal{H}}(K_{\gamma}(\epsilon)) \succeq \text{exp}\Big(-C_{\alpha}\gamma^{\frac{\alpha(\alpha - 1)}{(\alpha - 2)}}\epsilon^{-2\alpha}\Big)
\end{equation*}
and $C_{\alpha} > 0$ is some constant depending only on $\alpha > 0$. As most the details follow from \cite{talagrand1993new}, we sketch the argument here. Let $\{\epsilon_r\}_{r \geq 1}$ be a sequence such that $\frac{\epsilon}{2} = \sum_{r \geq 1} \epsilon_r$. For $r \geq 1$, let $S_r$ be a $2^{-r}$-net of $\mathcal{X}$ in the metric $d_{\gamma}$ and let $S_0 = \{x_0\}$ for some fixed $x_0 \in \mathcal{X}$. For $x \in \mathcal{X}$, define $\psi_r(x) \in S_r$  by $\psi_{r}(x) = \text{arg} \min_{y \in S_{r}} d_{\gamma}(x, y)$, and observe $\psi_r(x) \to x$ as $r \to \infty$. Then, we may write that:
\begin{align*}
    \lambda_{\mathcal{H}}(\epsilon K_{\gamma}) & = \lambda_{\mathcal{H}}\Big(\sup_{x \in \mathcal{X}} \langle h, \gamma e^{\gamma}_{x} \rangle \leq \epsilon\Big) \\
    & \geq \lambda_{\mathcal{H}}\Big(\bigcap_{r \geq 1} \{\sup_{y \in S_r}|\langle h, \gamma (e^{\gamma}_{y} - e^{\gamma}_{\psi_{r - 1}(y)}) \rangle| \leq \epsilon_r\}, |\langle h, \gamma e^{\gamma}_{x_0} \rangle| \leq \frac{\epsilon}{2}\Big) 
\end{align*}
Let $\Phi$ be the survival function of a standard normal (i.e. $1 - \Phi$ is the normal CDF). Hence, by the gaussianity of $\lambda$ and an application of Sidak's inequality \cite{sidak1968multivariate} (see also Lemma 4.1 in \cite{talagrand1993new}), we may write:
\begin{align}
    \lambda_{\mathcal{H}}(\epsilon K_{\gamma})) & \geq \lambda_{\mathcal{H}}\Big(\bigcap_{r \geq 1} \{\sup_{y \in S_r}|\langle h, \gamma (e^{\gamma}_{y} - e^{\gamma}_{\psi_{r - 1}(y)}) \rangle| \leq \epsilon_r\}, |\langle h, \gamma e^{\gamma}_{x_0} \rangle| \leq \frac{\epsilon}{2}\Big)  \nonumber \\
    & \geq (1 - 2\Phi(\epsilon(2\sqrt{\gamma})^{-1}))\prod_{r = 1}^{\infty} (1 - 2\Phi(2^{r-1}\epsilon_r))^{\mathcal{N}(\mathcal{X}, d_{\gamma}, 2^{-r})} \label{eq: Apply Max Variance} \\
    & \geq (1 - 2\Phi(\epsilon(2\sqrt{\gamma})^{-1}))\prod_{r = 1}^{\infty} (1 - 2\Phi(2^{r-1}\epsilon_r))^{K_{\alpha}\gamma^{\frac{\alpha(\alpha - 1)}{(\alpha - 2)}}2^{2\alpha r}} \label{eq: Apply Covering Number Bound}
\end{align}
for some $K_{\alpha} > 0$ depending only $\alpha$. Here, in \eqref{eq: Apply Max Variance} we have used $\text{Var}(\langle h, \gamma (e^{\gamma}_{y} - e^{\gamma}_{\psi_{r - 1}(y)})\rangle) = \gamma^2 \mathcal{E}(e^{\gamma}_{y} - e^{\gamma}_{\psi_{r - 1}(y)}, e^{\gamma}_{y} - e^{\gamma}_{\psi_{r - 1}(y)}) = d^2_{\gamma}(y, \psi_{r - 1}(y)) \leq 2^{2 - 2r}$ and $\text{Var}(\langle h, \gamma e^{\gamma}_{x_0}\rangle) = \mathcal{E}(\gamma e^{\gamma}_{x_0}, \gamma e^{\gamma}_{x_0}) = \gamma$ and in \eqref{eq: Apply Covering Number Bound} we have applied Lemma \ref{Modulus of Continuity} and volume doubling, which gives that:
\begin{equation*}
    \mathcal{N}(\mathcal{X}, d_{\gamma}, \epsilon) \preceq \gamma^{\frac{\alpha(\alpha - 1)}{(\alpha - 2)}}\epsilon^{-2\alpha}
\end{equation*}
Indeed, for $\delta > 0$, by the doubling property of $\mu$ there exists $m \geq \mathcal{N}(\mathcal{X}, d, \gamma^{\frac{1 - \alpha}{\alpha - 2}}\epsilon^2)$ such that $\Big\{B\Big(x_i, \gamma^{\frac{1 - \alpha}{\alpha - 2}}\delta^2\Big)\Big\}_{i = 1}^m$ covers $\mathcal{X}$ and$\Big\{B\Big(x_i, 0.2\gamma^{\frac{1 - \alpha}{\alpha - 2}}\delta^2\Big)\Big\}_{i = 1}^m$ are disjoint. Hence, we have that:
\begin{align}
    \mathcal{N}(\mathcal{X}, d_{\gamma}, \delta) & \preceq \mathcal{N}(\mathcal{X}, d, \gamma^{\frac{1 - \alpha}{\alpha - 2}}\delta^2)  \label{eq: Apply Distance Comparison} \\
    & \preceq \max_{i \in [m]} V(x_i, \gamma^{\frac{1 - \alpha}{\alpha - 2}}\delta^2)^{-1} \nonumber \\
    & \preceq \gamma^{\frac{\alpha(\alpha - 1)}{\alpha - 2}}\delta^{-2\alpha} \label{eq: Apply RVD}
 \end{align}
where \eqref{eq: Apply Distance Comparison} follows from Lemma \ref{Modulus of Continuity} and \eqref{eq: Apply RVD} follows from \eqref{eq: VD}. It remains to choose the sequence $\{\epsilon_r\}_{r \geq 0}$. Let $p \in \mathbb{N}$ be such that:
\begin{equation*}
    2^{-p} \leq \epsilon < 2^{-p + 1}
\end{equation*}
and choose $\epsilon_r = C2^{-p - \frac{|r-p|}{2}}$, where $C \in (0, 1)$ is an absolute constant chosen appropriately so $\sum_{r \geq 1} \epsilon_r < 2^{-p - 1} \leq \frac{\epsilon}{2}$. Then, we have:
\begin{align*}
    \lambda_{\mathcal{H}}(\epsilon K_{\gamma}) & \geq \lambda_{\mathcal{H}}(2^{-p} K_{\gamma}) \\
    & \geq (1 - 2\Phi(2^{-p}(2\sqrt{\gamma})^{-1})) \prod_{r = 1}^{\infty} (1 - 2\Phi(C2^{r-p - \frac{|r - p|}{2} - 1}))^{\gamma^{\frac{\alpha(\alpha - 1)}{(\alpha - 2)}}2^{2\alpha r}} 
\end{align*}
Noting the standard inequalities:
\begin{align*}
    1 - 2\Phi(a) \geq La \hspace{2mm} a \leq 2 \\
    1 - 2\Phi(a) \geq \text{exp}\Big(-2e^{-\frac{a^2}{2}}\Big)   \hspace{2mm} a \geq 1
\end{align*}
for an absolute constant $L > 0$ (see e.g. pg. 520 of \cite{talagrand1993new}). Hence, we obtain, for an absolute constant $C' = \frac{2\log 2 - 2\log C}{\log 2}$:
\begin{align*}
    \log \lambda_{\mathcal{H}}(K_{\gamma}(\epsilon)) & \succeq \log (L2^{-p-1}\gamma^{-\frac{1}{2}}) + \sum_{r = 1}^{p + C'} \gamma^{\frac{\alpha(\alpha - 1)}{(\alpha - 2)}}2^{2\alpha r}\log(L2^{-\frac{3|r - p|}{2}-1})  -2 \sum_{r > p + C'} \gamma^{\frac{\alpha(\alpha - 1)}{(\alpha - 2)}}2^{2\alpha r} e^{-C^22^{r - p - 2}} \\
    & \succeq \log (L2^{-p-1}\gamma^{-\frac{1}{2}}) + \gamma^{\frac{\alpha(\alpha - 1)}{(\alpha - 2)}}2^{2\alpha p}\Big(\sum_{r = 1}^{p + C'} 2^{2\alpha(r-p)}\log(L2^{-\frac{3|r - p|}{2} - 1}) -2 \sum_{r > p + C'} 2^{2\alpha(r-p)} e^{-C^22^{r - p - 2}}\Big) \\
    & \succeq C_{\alpha}\gamma^{\frac{\alpha(\alpha - 1)}{(\alpha - 2)}}2^{2\alpha p} + \log (L2^{-p-1}\gamma^{-\frac{1}{2}}) \\
    & \succeq C_{\alpha}\gamma^{\frac{\alpha(\alpha - 1)}{(\alpha - 2)}}2^{2\alpha p}
\end{align*}
for some constant $C_{\alpha} < 0$ depending only on $\alpha$ (as we suppose $\gamma > 0$ to be sufficiently large). Now, noting that $\epsilon < 2^{-p + 1}$, we have:
\begin{equation*}
    \log \lambda_{\mathcal{H}}(K_{\gamma}(\epsilon))  \geq C_{\alpha}\gamma^{\frac{\alpha(\alpha - 1)}{(\alpha - 2)}}\epsilon^{-2\alpha}
\end{equation*}
where we have redefined $C_{\alpha} < 0$ appropriately. Now, applying Theorem 1 in \cite{kuelbs1993metric} we obtain:
\begin{equation*}
    \log \mathcal{N}(\mathbb{B}(\mathcal{H}), K_{\gamma}(\epsilon)) \preceq \gamma^{\frac{\alpha(\alpha - 1)}{(\alpha + 1)(\alpha - 2)}}\epsilon^{-\frac{2\alpha}{\alpha + 1}}
\end{equation*}
where $\mathcal{N}(\mathbb{B}(\mathcal{H}), K_{\gamma}(\epsilon))$ denotes the covering number of the unit Dirichlet ball $\mathbb{B}(\mathcal{H})$ by copies of the convex body $K_{\gamma}(\epsilon)$. By the definition of $K_{\gamma}(\epsilon)$ and $\mathcal{H}_{\gamma}$, the latter result becomes:
\begin{equation*}
    \log \mathcal{N}(\epsilon, \mathcal{H}_{\gamma}, \|\cdot\|_{\infty}) \preceq \gamma^{\frac{\alpha(\alpha - 1)}{(\alpha + 1)(\alpha - 2)}}\epsilon^{-\frac{2\alpha}{\alpha + 1}}
\end{equation*}
\end{proof}

\begin{proof}[Proof of Theorem \ref{ERM Risk}]
Equipped with the metric entropy estimate from Proposition \ref{Covering Numbers}, we follow the standard procedure to derive an oracle inequality for the excess risk (see e.g \cite{koltchinskii2011oracle, bartlett2002localized}). Observe that the class $\mathcal{H}_{\gamma}$ is convex and fix:
\begin{equation*}
    \bar{f} \in \text{arg} \min_{f \in \mathcal{H}_{\gamma_n}} \|f - f^{*}\|_{L^2(\nu)}
\end{equation*}
Recall that $||f||_{\infty} \leq \sqrt{\gamma_n}$ for all $f \in \mathcal{H}_{\gamma_n}$ (by definition \eqref{eq: Reg Function Class} and Cauchy-Schwarz). Define the excess loss $\ell_{f}$ and loss classes $\mathcal{F}_{\gamma_n, r}$ and $\tilde{\mathcal{F}}_{\gamma_n, r}$:
\begin{align*}
    \ell_f(x) & = (f^{*}(x) - f(x))^2 - (f^{*}(x) - \bar{f}(x))^2 \hspace{4mm} \forall f \in \mathcal{H}_{\gamma_n}\\
    \mathcal{F}_{\gamma_n, r} & = \{\ell_f: f \in \mathcal{H}_{\gamma_n} \cap r\mathbb{B}(\bar{f}, L^2(\nu))\} \\
    \tilde{\mathcal{F}}_{\gamma_n, r} & = \{g_f(x, y) \equiv (y - f^{*}(x))(\bar{f}(x) - f(x)): f \in \mathcal{H}_{\gamma_n} \cap r\mathbb{B}(\bar{f}, L^2(\nu))\}
\end{align*}
Then, recalling $\epsilon = y - f^{*}(x)$, we may write:
\begin{align}
    (y - f(x))^2 - (y - \bar{f}(x))^2 & = 2\epsilon(\bar{f}(x) - f(x)) + (f^{*} - f(x))^2 - (f^{*} - \bar{f}(x))^2 \nonumber \\
    & = 2g_f(x, y) + \ell_{f}(x) \label{eq: Excess Loss}
\end{align}
Let $r^2_n = \sup \{\frac{1}{n}\sum_{i = 1}^n (f - \bar{f})^2(X_i): f \in \mathcal{H}_{\gamma_n} \cap r\mathbb{B}(\bar{f}, L^2(\nu))\}$ and $\tilde{r}^2_n = \sup \{\frac{1}{n}\sum_{i = 1}^n \epsilon^2_i(f - \bar{f})^2(X_i): f \in \mathcal{H}_{\gamma_n} \cap r\mathbb{B}(\bar{f}, L^2(\nu))\}$. Since $\epsilon \sim \text{SG}(\rho^2)$ and $||f||_{\infty} \leq \sqrt{\gamma_n}$ for all $f \in \mathcal{H}_{\gamma_n}$, it follows from Adamczak's version of Talagrand's inequality (Theorem 4 in \cite{adamczak2008tail}) and classical Dudley chaining, that with probability $1-\delta$, for all $f \in \mathcal{H}_{\gamma_n} \cap r\mathbb{B}(\bar{f}, L^2(\nu))$:
\begin{align}
    \frac{1}{n}\sum_{i = 1}^n \epsilon_i (\bar{f}(X_i) - f(X_i)) & \preceq \mathbb{E}\Big[\sup_{g \in \tilde{\mathcal{F}}_{\gamma_n, r}} \frac{1}{n}\sum_{i = 1}^n g(X_i, Y_i)\Big] + \sqrt{\frac{\rho^2 r^2 \log \Big(\frac{1}{\delta}\Big)}{n}} +  \frac{\log \Big(\frac{1}{\delta}\Big) \rho \log n \sqrt{\gamma_n}}{n} \nonumber \\
    & \preceq \mathbb{E}\Big[\frac{1}{\sqrt{n}}\int_{0}^{\tilde{r}_n} \sqrt{\log \mathcal{N}(s, \tilde{\mathcal{F}}_{\gamma_n, r}, \|\cdot\|_{L^2(\mathbb{P}_n)})}ds\Big] + \sqrt{\frac{\rho^2 r^2 \log \Big(\frac{1}{\delta}\Big)}{n}} +  \frac{\log \Big(\frac{1}{\delta}\Big) \rho \log n \sqrt{\gamma_n}}{n} \nonumber \\
    & \preceq \mathbb{E}\Big[\frac{\max_{i \in [n]} |\epsilon_i|}{\sqrt{n}}\int_{0}^{r_n} \sqrt{\log \mathcal{N}(s, \mathcal{H}_{\gamma_n}, \|\cdot\|_{\infty})}ds\Big] + \sqrt{\frac{\rho^2 r^2 \log \Big(\frac{1}{\delta}\Big)}{n}} \label{eq: Change of Variable} \\
    & \hspace{2mm} +  \frac{\log \Big(\frac{1}{\delta}\Big) \rho \log n \sqrt{\gamma_n}}{n}  \nonumber \\
    & \preceq \frac{\rho \sqrt{\log n}}{\sqrt{n}}\int_{0}^{\sqrt{\mathbb{E}[r_n^2]}} \sqrt{\log \mathcal{N}(s, \mathcal{H}_{\gamma_n}, \|\cdot\|_{\infty})}ds + \sqrt{\frac{\rho^2 r^2 \log \Big(\frac{1}{\delta}\Big)}{n}} +  \frac{\log \Big(\frac{1}{\delta}\Big) \rho \log n \sqrt{\gamma_n}}{n} \label{eq: Dudley Concavity} \\
    & \preceq \frac{\rho \sqrt{\log n} \gamma_n^{\frac{\alpha(\alpha - 1)}{2(\alpha + 1)(\alpha - 2)}}r^{\frac{1}{\alpha + 1}}}{\sqrt{n}} +\sqrt{\frac{\rho^2 r^2 \log \Big(\frac{1}{\delta}\Big)}{n}} +  \frac{\log \Big(\frac{1}{\delta}\Big) \rho \log n \sqrt{\gamma_n}}{n} \label{eq: Apply Metric Entropy} 
\end{align}
where \eqref{eq: Change of Variable} follows from observing $\tilde{r}_n \leq \max_{i \in [n]} |\epsilon_i| \cdot r_n$ and a change of variable, \eqref{eq: Dudley Concavity} follows from the independence of the subgaussian $\text{SG}(\rho^2)$ noise $\{\epsilon_i\}_{i = 1}^n$ and $\{X_i\}_{i = 1}^n$, and the concavity of the Dudley integral and \eqref{eq: Apply Metric Entropy} follows from Proposition \ref{Covering Numbers}. 
Moreover, for $f_1, f_2 \in \mathcal{H}_{\gamma_n}$, we have:
\begin{equation*}
    |\ell_{f_1}(x) - \ell_{f_2}(x)| = |(2f^{*}(x) - f_1(x) - f_2(x))(f_1(x) - f_2(x))| \leq 4\sqrt{\gamma_n}|f_1(x) - f_2(x)|
\end{equation*}
where we have again used the fact that $||f||_{\infty} \leq \sqrt{\gamma_n}$ for all $f \in \mathcal{H}_{\gamma_n}$ and assumed, w.l.o.g, that $\|f^{*}\|_{\infty} \leq \gamma_n$ (since $\gamma_n \to \infty$ as $n \to \infty$). Hence, by the contraction principle, we obtain the following estimate for the expected localized Rademacher complexity: 
\begin{align}
    \mathbb{E}[\mathcal{R}_n(\mathcal{F}_{\gamma_n, r})] & \preceq \sqrt{\gamma_n}\mathbb{E}[\mathcal{R}_n(\mathcal{H}_{\gamma_n} \cap r\mathbb{B}(\bar{f}, L^2(\nu)))] \nonumber \\
    & \preceq \sqrt{\frac{r^{\frac{2}{\alpha + 1}}\gamma_n^{\frac{2\alpha^2 - 2\alpha - 2}{(\alpha + 1)(\alpha - 2)}}}{n}}  \equiv \psi(r) \label{eq: Rademacher and Metric Entropy}
\end{align}
where \eqref{eq: Rademacher and Metric Entropy} again follows from a Dudley upper bound and Proposition \ref{Covering Numbers}. From \eqref{eq: Rademacher and Metric Entropy} and Bousquet's version of Talagrand's inequality \cite{bousquet2002bennett}, we have that:
\begin{equation}
\label{eq: Square Deviation}
    \sup_{f \in \mathcal{H}_{\gamma_n} \cap r\mathbb{B}(\bar{f}, L^2(\nu))} \mathbb{P}_n \ell_f  - P \ell_f \preceq \psi(r) + r\sqrt{\frac{\gamma_n\log \Big(\frac{1}{\delta}\Big)}{n}} + \frac{\gamma_n \log \Big(\frac{1}{\delta}\Big)}{n}
\end{equation}
with probability $1 - \delta$. Hence, from \eqref{eq: Excess Loss}, \eqref{eq: Apply Metric Entropy}, and \eqref{eq: Square Deviation}, we have, with probability $1 - 2\delta$, that for all $f \in \mathcal{H}_{\gamma_n} \cap r\mathbb{B}(\bar{f}, L^2(\nu))$:
\begin{align}
    (\mathbb{P}_n - P)((y - f(x))^2 - (y - \bar{f}(x))^2) & = (\mathbb{P}_n - P)(2\epsilon(\bar{f}(x) - f(x)) + \ell_{f}(x)) \nonumber \\
    & \preceq \psi(r) + \rho\sqrt{\frac{r^{\frac{2}{\alpha + 1}}\gamma_n^{\frac{\alpha(\alpha - 1)}{(\alpha + 1)(\alpha - 2)}}\log n }{n}} + \sqrt{\frac{r^2\gamma_n\log \Big(\frac{1}{\delta}\Big)}{n}} + \frac{\rho \gamma_n \log \Big(\frac{1}{\delta}\Big) \log n}{n} \nonumber \\
    & \preceq \psi(r) + r\sqrt{\frac{\gamma_n\log \Big(\frac{1}{\delta}\Big)}{n}} + \frac{\rho \gamma_n \log \Big(\frac{1}{\delta}\Big) \log n}{n} \label{eq: Order Precedence}
\end{align}
where \eqref{eq: Order Precedence} follows from the fact that:
\begin{align*}
    \rho\sqrt{\frac{r^{\frac{2}{\alpha + 1}}\gamma_n^{\frac{\alpha(\alpha - 1)}{(\alpha + 1)(\alpha - 2)}}\log n }{n}} & = \rho \sqrt{\log n}\gamma_n^{\frac{2-\alpha(\alpha - 1)}{2(\alpha + 1)(\alpha - 2)}} \cdot \frac{r^{\frac{1}{\alpha + 1}}\gamma_n^{\frac{2\alpha^2 - 2\alpha - 2}{2(\alpha + 1)(\alpha - 2)}}}{\sqrt{n}}  \\
    & = \rho \sqrt{\log n}\gamma_n^{\frac{2-\alpha(\alpha - 1)}{2(\alpha + 1)(\alpha - 2)}} \cdot \psi(r) \\
    & \preceq \psi(r)
\end{align*}
since $\gamma_n \asymp n^{\frac{\alpha - 2}{2\alpha}}$ and $\alpha \geq 3$. We observe that the solution $r^{*}$ of $\psi(r) \asymp r^2$ is given by:
\begin{equation*}
    r^{*} \asymp \sqrt{\frac{\gamma_n^{\frac{2\alpha^2 - 2\alpha - 2}{(2\alpha + 1)(\alpha - 2)}}}{n^{\frac{\alpha + 1}{2\alpha + 1}}}}
\end{equation*}
Then, by Theorem 5.1 in \cite{koltchinskii2011oracle} (see also Theorem 3.3 and 6.1 in \cite{bartlett2002localized}), we have that with probability $1 - 2\delta$:
\begin{equation}
\label{eq: Excess Risk Bound 1}
    \mathbb{E}_{P}[(Y - \hat{f}_D(X))^2] \preceq \mathbb{E}_{P}[(Y - \bar{f}(X))^2] + \frac{\gamma_n^{\frac{2\alpha^2 - 2\alpha - 2}{(2\alpha + 1)(\alpha - 2)}}}{n^{\frac{\alpha + 1}{2\alpha + 1}}} + \frac{\rho \gamma_n \log \Big(\frac{1}{\delta}\Big) \log n}{n}
\end{equation}
Let $f^{*}_{\gamma_n} \in \mathcal{H}_{\gamma_n}$ be given by $f^{*}_{\gamma_n}(x) = \mathcal{E}(f^{*}, \gamma_n e_{x}^{\gamma_n})$. Then, by definition of $\bar{f}$ we have that:
\begin{equation*}
    ||\bar{f} - f^{*}||^2_{L^2(\nu)} \leq ||f^{*}_{\gamma_n} - f^{*}||^2_{L^2(\nu)} \leq \gamma_n^{\frac{2}{2 - \alpha}}\mathcal{E}(f^{*}, f^{*})
\end{equation*}
by Theorem \ref{Capacitary Poincare}. Substituting back into \eqref{eq: Excess Risk Bound 1} and applying the independence of $Y - f^{*}(X)$ and $X$, we obtain:
\begin{equation}
\label{eq: Excess Risk Bound 2}
    \|\hat{f}_D - f^{*}\|^2_{L^2(\nu)} \preceq \gamma_n^{\frac{2}{2 - \alpha}}\mathcal{E}(f^{*}, f^{*}) + \frac{\gamma_n^{\frac{2\alpha^2 - 2\alpha - 2}{(2\alpha + 1)(\alpha - 2)}}}{n^{\frac{\alpha + 1}{2\alpha + 1}}} + \frac{\rho \gamma_n \log \Big(\frac{1}{\delta}\Big)\log n}{n}
\end{equation}
Substituting the choice of $\gamma_n \asymp n^{\frac{\alpha - 2}{2\alpha}}$ (which can also be viewed as optimizing over $\gamma_n$), we obtain our result. 
\end{proof}
\section{Auxiliary Estimates}
\label{Auxiliary Estimates}

\begin{lemma}
\label{Regularized Potential Norm}
\begin{equation*}
    \gamma^2_n||(C_{\nu} + \lambda)^{-\frac{1}{2}}e_{x, n}||^2_{\mathcal{H}} \leq \frac{\gamma_n}{\lambda}
\end{equation*}
\end{lemma}
\begin{proof}
This follows directly by applying Cauchy Schwarz and noting that $\|(C_{\nu} + \lambda)^{-\frac{1}{2}}\|^2_{\mathcal{H} \to \mathcal{H}} \leq \lambda^{-1}$ and $\|e_{x, n}\|^2_{\mathcal{H}} = \gamma^{-1}_n = \text{cap}(\mathcal{O}_{x, n})$
\end{proof}
\begin{lemma}
\label{Operator Approximation Error}
There exists a $K > 1$ (independent of $n$ and $\lambda$) such that:
\begin{equation*}
    ||(C_{\nu} + \lambda)^{-\frac{1}{2}}(\hat{C}_{\nu} - C_{\nu})(C_{\nu} + \lambda)^{-\frac{1}{2}}||_{\mathcal{H} \to \mathcal{H}} \leq \sqrt{\frac{K^2\gamma^{\frac{2\beta}{\beta - \alpha}}_n}{\lambda^2} + \frac{K\gamma_n^{\frac{\beta}{\beta - \alpha}}}{\lambda}}
\end{equation*}
as $n \to \infty$. 
\end{lemma}
\begin{proof}
Let $h \in \mathcal{H}$. Then, we have:
\begin{align}
    \mathcal{E}(h, (\hat{C}_{\nu} - C_{\nu})h) & = \int_{\mathcal{X}} ((P^{\gamma_n}_{x}h)^2 - h^2(x))d\nu(x) \nonumber \\
    & = \sqrt{\Big(\int_{\mathcal{X}} (P^{\gamma_n}_{x}h +  h(x))^2 d\nu(x)\Big)\Big(\int_{\mathcal{X}} (P^{\gamma_n}_{x}h -  h(x))^2 d\nu(x)\Big)} \label{eq: Apply Cauchy-Schwarz}\\ 
    & \preceq \sqrt{(\gamma_n^{\frac{\beta}{\beta - \alpha}}\mathcal{E}(h, h) + ||h||^2_{L^2(\nu)})(\gamma_n^{\frac{\beta}{\beta - \alpha}}\mathcal{E}(h, h))} \label{eq: Apply Capacitary Lemmas}
\end{align}
where in \eqref{eq: Apply Cauchy-Schwarz} we have applied the Cauchy-Schwarz inequality, and in \eqref{eq: Apply Capacitary Lemmas} we have applied Theorem \ref{Capacitary Poincare} and Lemma \ref{Average Projection}. Now, setting $h = (C_{\nu} + \lambda)^{-\frac{1}{2}}g$ for $g \in \mathcal{H}$, and noting that:
\begin{equation*}
    ||(C_{\nu} + \lambda)^{-\frac{1}{2}}g||^2_{L^2(\nu)} = ||C^{\frac{1}{2}}_{\nu}(C_{\nu} + \lambda)^{-\frac{1}{2}}g||^2_{\mathcal{H}} \leq ||g||^2_{\mathcal{H}}
\end{equation*}
we obtain our result. 
\end{proof}

\begin{lemma}
\label{Sampling Approximation Error}
For $h \in \mathcal{H} \cap L^2(\nu)$
\begin{equation*}
    ||(C_{\nu} + \lambda)^{-\frac{1}{2}}(\hat{C}_{\nu} - C_{\nu})h||_{\mathcal{H}} \preceq \sqrt{\Big(\frac{\gamma_n^{\frac{\beta}{\beta - \alpha}}}{\lambda} \vee 1\Big)\gamma_n^{\frac{\beta}{\beta - \alpha}}||h||^2_{\mathcal{H}} + \frac{\gamma_n^{\frac{\beta}{\beta - \alpha}}}{\lambda} ||h||^2_{L^2(\mathcal{X}, \nu)}}
\end{equation*}
as $n \to \infty$
\end{lemma}

\begin{proof}
We proceed similarly to the proof of Lemma \ref{Operator Approximation Error}. First, we observe that, for any $g \in \mathcal{H}$:

\begin{align}
    \mathcal{E}(g, (\hat{C}_{\nu} - C_{\nu})h) & = \int_{\mathcal{X}} [g(x)(P^{\gamma_n}_{x}h  - h(x)) + P^{\gamma_n}_{x}h(P^{\gamma_n}_{x}g  - g(x))]d\nu(x) \nonumber \\
    & \leq \sqrt{\int_{\mathcal{X}} g^2(x) d\nu(x)\int_{\mathcal{X}} (P^{\gamma_n}_{x}h -  h(x))^2 d\nu(x)} + \sqrt{\int_{\mathcal{X}} (P^{\gamma_n}_{x}h)^2 d\nu(x)\int_{\mathcal{X}} (P^{\gamma_n}_{x}g -  g(x))^2 d\nu(x)} \label{eq: Apply Cauchy-Schwarz 2}\\ 
    & \preceq ||g||_{L^2(\mathcal{X}, \nu)}\sqrt{\gamma_n^{\frac{\beta}{\beta - \alpha}} \mathcal{E}(h, h)} + \sqrt{\gamma_n^{\frac{\beta}{\beta - \alpha}} \mathcal{E}(g, g)\Big(\gamma_n^{\frac{\beta}{\beta - \alpha}} \mathcal{E}(h, h) + ||h||^2_{L^2(\mathcal{X}, \nu)}\Big)} 
    \label{eq: Apply Capacitary Lemmas 2}
\end{align}
where again \eqref{eq: Apply Cauchy-Schwarz 2} follows from Cauchy-Schwarz and \eqref{eq: Apply Capacitary Lemmas 2} follows from the application of Theorem \ref{Capacitary Poincare} and Lemma \ref{Average Projection}. Writing $g = (C_{\nu} + \lambda)^{-\frac{1}{2}}\tilde{g}$ for some $\tilde{g} \in \mathcal{H}$ with $\|\tilde{g}\|_{\mathcal{H}} = 1$, recalling that $\mathcal{E}(h, h) = \|\cdot\|^2_{\mathcal{H}}$, and $\gamma^{-1}_n, \lambda \to 0$ as $n \to \infty$, we obtain our result. 
\end{proof}

\begin{lemma}
\label{Average Regularized Projection}
For any $h \in \mathcal{H}$
\begin{equation*}
   \Big\|\int \gamma_n(C_{\nu} + \lambda)^{-\frac{1}{2}}e_{x, n}(h(x) - P_{x}^{\gamma_n}h)d\nu(x)\Big\|_{\mathcal{H}} \preceq ||h||_{\mathcal{H}}\sqrt{\frac{\gamma_n^{\frac{2\beta}{\beta - \alpha}}}{\lambda} + \gamma_n^{\frac{\beta}{\beta - \alpha}}}
\end{equation*}
as $n \to \infty$
\end{lemma}
\begin{proof}
Note for any $g \in \mathcal{H}$, we have that:
\begin{align}
    \mathcal{E}\Big(g, \int \gamma_n e_{x, n} (h(x) - P_{x}^{\gamma_n}h) d\nu(x)\Big) 
    & = \int P_{x}^{\gamma_n}g (h(x) - P_{x}^{\gamma_n}h) d\nu(x) \nonumber \\
    & \leq \sqrt{\Big(\int (P_{x}^{\gamma_n}g)^2 d\nu(x)\Big) \Big(\int (h(x) - P_{x}^{\gamma_n}h)^2 d\nu(x)\Big)} \nonumber \\
    & \leq \sqrt{\gamma_n^{\frac{\beta}{\beta - \alpha}} \mathcal{E}(h, h)\Big(\gamma_n^{\frac{\beta}{\beta - \alpha}} \mathcal{E}(g, g) + ||g||^2_{L^2(\mathcal{X}, \nu)}\Big)} \label{eq: Apply Capacitary Lemmas 3}
\end{align}
where in \eqref{eq: Apply Capacitary Lemmas 3} we have again applied Theorem \ref{Capacitary Poincare} and Lemma \ref{Average Projection}. Now, writing $g = (C_{\nu} + \lambda)^{-\frac{1}{2}}\tilde{g}$ for some $\tilde{g} \in \mathcal{H}$ with $||g||_{\mathcal{H}} = 1$, we obtain, as in Lemma \ref{Sampling Approximation Error}, our result. 
\end{proof}


\bibliographystyle{imsart-number} 
\bibliography{ref}       


\end{document}